\documentclass[11pt,a4paper,leqno]{amsart}
\usepackage{amssymb, amsmath}
\newtheorem{theorem}{Theorem}[section]

\newtheorem{proposition}[theorem]{Proposition}
\newtheorem{remark}[theorem]{Remark}

\newcommand{\dt}{\mathfrak{d}}
\newcommand{\tr}{\mathfrak{t}}
\newcommand{\bach}{\mathfrak{B}}
\newcommand{\ce}{\mathcal{E}}

\newcommand{\affcon}{D} 

\begin{document}
\title[Bach-flat isotropic gradient Ricci solitons]
{Bach-flat isotropic gradient Ricci solitons}
\author[Calvi\~{n}o-Louzao, Garc\'ia-R\'io, Guti\'errez-Rodr\'\i guez, V\'{a}zquez-Lorenzo]{E. Calvi\~{n}o-Louzao, E. Garc\'ia-R\'io, I. Guti\'errez-Rodr\'\i guez,   \\
R. V\'{a}zquez-Lorenzo}
\address{(E. C.-L.) Conseller\'\i a de Cultura, Educaci\'on e Ordenaci\'on Universitaria, Edificio Administrativo San Caetano, 15781 Santiago de Compostela, Spain}
\address{(E. G.-R.) Faculty of Mathematics,
University of Santiago de Compostela, 15782 Santiago de Compostela, Spain}
\address{(I. G.-R.) Faculty of Mathematics,
University of Santiago de Compostela, 15782 Santiago de Compostela, Spain}
\address{(R. V.-L.) Department of Mathematics, IES de Ribadeo Dionisio Gamallo, Ribadeo, Spain}
\email{{estebcl@edu.xunta.es}, {eduardo.garcia.rio@usc.es}, {dzohararte@hotmail.fr},
{ravazlor@edu.xunta.es}}
\subjclass[2010]{53C25, 53C20, 53C44}
\date{}
\keywords{Gradient Ricci soliton, Bach tensor, Riemannian extension, Affine surface}

\begin{abstract}
We construct examples of Bach-flat gradient Ricci solitons which are neither half conformally flat nor conformally Einstein.
\end{abstract}

\maketitle

\section{Introduction}

Let $(M,g)$ be a pseudo-Riemannian manifold. Let $f\in\mathcal{C}^\infty(M)$. We say that $(M,g,f)$ is
a \emph{gradient Ricci soliton} if
the following equation is satisfied:
\begin{equation}\label{eq:ricci-soliton}
\operatorname{Hes}_f+\rho=\lambda\, g\,,
\end{equation}
for some $\lambda\in\mathbb{R}$, where $\rho$ is the \emph{Ricci tensor}, and
$\operatorname{Hes}_f=\nabla df$ is the \emph{Hessian tensor} acting on $f$.

A gradient Ricci soliton is said to be {\it trivial} if the potential function $f$ is constant, since equation \eqref{eq:ricci-soliton} reduces to the Einstein equation $\rho=\lambda g$. Besides being a generalization of Einstein manifolds, the main interest of gradient Ricci solitons comes from the fact that they correspond to self-similar solutions of the {\it Ricci flow} $\,\partial_tg(t)=-2\rho_{g(t)}$.
Ricci solitons are 
 ancient solutions of the flow in the shrinking case ($\lambda>0$), eternal solutions in the steady case ($\lambda=0$), and immortal solutions in the expanding case ($\lambda<0$).
Gradient Ricci solitons have been extensively
investigated in the literature (see for example the discussion in
\cite{Cao} and the references therein). 
Classifying gradient Ricci solitons under geometric conditions is a problem of special interest.

The gradient Ricci soliton equation codifies geometric information of $(M,g)$ in terms of the Ricci curvature and the second fundamental form of the level sets of the potential function $f$. 
The fact that the Ricci tensor completely determines the curvature tensor in the locally conformally flat case has yielded some results in this situation \cite{CC2, MS, PW}.
Any locally conformally flat gradient Ricci soliton is locally a warped product in the Riemannian setting \cite{FLGR}. The higher signature case, however, allows other possibilities when the level sets of the potential function are degenerate hypersurfaces \cite{BGG}.
Generalizing the locally conformally flat condition, four-dimensional half conformally flat (i.e., self-dual or anti-self-dual) gradient Ricci solitons  have been investigated in the Riemannian and neutral signature cases \cite{BVGR, chen-wang}. While they are locally conformally flat in the Riemannian situation, neutral signature allows other examples given by Riemannian extensions of affine gradient Ricci solitons.

Let $W$ be the Weyl conformal curvature tensor of $(M,g)$.
The Bach tensor,
$
\bach_{ij}=\nabla^k\nabla^\ell W_{kij\ell}+\frac{1}{2}\rho^{k\ell}W_{kij\ell}\,
$,
is conformally invariant in dimension four. 
Bach-flat metrics contain half conformally flat and conformally Einstein metrics as special cases \cite{Besse}.
Hence, a natural problem is to classify Bach-flat gradient Ricci solitons. The Riemannian case was investigated in \cite{GRS1,GRS2} both in the shrinking and steady cases. In all situations the Bach-flat condition reduces to the locally conformally flat one under some natural conditions.

Our main purpose in this paper is to construct new examples of Bach-flat gradient Ricci solitons. The corresponding potential functions have degenerate level set hypersurfaces and their underlying structure is never locally conformally flat, in sharp contrast with the Riemannian situation. 
These metrics are realized on the cotangent bundle $T^*\Sigma$ of an affine surface $(\Sigma,\affcon)$, and they may be viewed as perturbations of the classical Riemannian extensions introduced by Patterson and Walker in \cite{Patterson-Walker52}.

Here is a brief guide to some of the most important results of this paper. 
In Theorem \ref{th:Bach-zero} we show that, for any affine surface $(\Sigma,\affcon)$ admitting a parallel nilpotent $(1,1)$-tensor field $T$, the modified Riemannian extension $(T^*\Sigma,g_{\affcon,T,\Phi})$ is Bach-flat. Moreover we show that  Bach-flatness is independent of the deformation tensor field $\Phi$, thus providing an infinite family of Bach-flat metrics for any initial data $(\Sigma,\affcon,T)$.
Affine surfaces admitting a parallel nilpotent $(1,1)$-tensor field $T$ are characterized in Proposition \ref{prop:1} by the recurrence of the symmetric part of the Ricci tensor, being
$\operatorname{ker}T$ a parallel one-dimensional distribution whose integral curves are geodesics. This class of surfaces generalizes those considered in \cite{Opozda}.

The previous construction is used in Theorem \ref{th:solitons} to show that, for any smooth function $h\in\mathcal{C}^\infty(\Sigma)$, there exist appropriate deformation tensor fields $\Phi$ so that $(T^*\Sigma,g_{\affcon,T,\Phi},f=h\circ\pi)$ is a steady gradient Ricci soliton if and only if $dh(\operatorname{ker}T)=0$. This provides infinitely many examples of Bach-flat gradient Ricci solitons in neutral signature. 

Theorem \ref{th:sd-asd} and Theorem \ref{th:conformally-Einstein} show that $(T^*\Sigma,g_{\affcon,T,\Phi})$ is generically strictly Bach-flat, i.e., neither half conformally flat nor conformally Einstein. Moreover, Theorem \ref{th:sd-asd} is used in Proposition \ref{prop:4} to construct new examples of anti-self-dual metrics. Turning to gradient Ricci solitons, we show in Theorem \ref{prop:11} the existence of anti-self-dual steady gradient Ricci solitons which are not locally conformally flat.

The paper is organized as follows. Some basic results on the Bach tensor and gradient Ricci solitons are introduced in Section \ref{preliminaries}, as well as a sketch of the construction of modified Riemannian extensions $g_{\affcon,\Phi,T}$. We use these metrics in Section \ref{se:Bach-flat} to show that, for any parallel tensor field $T$ on $(\Sigma,\affcon)$,  $g_{\affcon,\Phi,T}$ is Bach-flat if and only if $T$ is either a multiple of the identity or nilpotent
(cf. Theorem \ref{th:Bach-zero}).
In Section \ref{se:4} we show that for each initial data $(\Sigma,\affcon,T)$ there are an infinite number of Bach-flat steady gradient Ricci solitons (cf. Theorem \ref{th:solitons}). Non-triviality of the examples is obtained after an examination of the half conformally flat condition (cf. Section~\ref{se:W-half}) and the conformally Einstein property (cf. Section~\ref{se:CE}) of the modified Riemannian extensions introduced in Section~\ref{preliminaries}. As a consequence, new anti-self-dual gradient Ricci solitons are exhibited in Theorem \ref{prop:11}. Finally, we specialize this construction in Section \ref{re:flat-11} to provide some illustrative examples.

\section{Preliminaries}\label{preliminaries}
Let $(M^n,g)$ be a pseudo-Riemannian manifold with Ricci curvature $\rho$ and scalar curvature $\tau$. 
Let $W$ denote the Weyl conformal curvature tensor and define
$W[\rho](X,Y)=\sum_{ij}\varepsilon_i\varepsilon_j W(E_i,X,Y,E_j)\rho(E_i,E_j)$, where $\{ E_i\}$ is a local orthonormal frame and $\varepsilon_i=g(E_i,E_i)$. Then the \emph{Bach tensor} is defined by (see \cite{Bach})
\begin{equation}
\label{eq:Bach-tensor}
\bach=\operatorname{div}_1\operatorname{div}_4 W+\frac{n-3}{n-2} W[\rho]\,,
\end{equation}
where $\operatorname{div}$ is the divergence operator.

Let $\mathfrak{S}=\rho -\frac{\tau}{2(n-1)}\, g$ denote the Schouten tensor of $(M,g)$. Let the \emph{Cotton tensor}, $\displaystyle\mathfrak{C}_{ijk}=(\nabla_i \mathfrak{S})_{jk} - (\nabla_j \mathfrak{S})_{ik}$\,; it provides a measure of the lack of symmetry on the covariant derivative of the  Schouten tensor. Since $\operatorname{div}_4 W=-\frac{n-3}{n-2}\,\mathfrak{C}$, the  Bach and the Cotton tensors of any four-dimensional manifold are related by
$\displaystyle
\bach$  $=$ $\frac{1}{2}\left(-\operatorname{div}_1\mathfrak{C}+W[\rho]\right)\,$.

The Bach tensor, which is trace-free and conformally invariant in dimension $n=4$,  has been broadly investigated in the literature, both from the geometrical and physical viewpoints (see for example \cite{CH, Derd, DuT} and references therein). It is the gradient of the $L^2$ functional of the Weyl curvature on compact manifolds. 
The field equations of conformal gravity are equivalent to setting the Bach tensor equal to zero
and it is also central in the study of the Bach flow, a geometric flow which is quadratic on the curvature and whose fixed points are the vacuum solutions of conformal Weyl gravity~\cite{BBLP}.

Besides the half conformally flat metrics and the conformally Einstein ones, there are few known examples of strictly Bach-flat manifolds, meaning the ones which are neither half conformally flat nor conformally Einstein (see, for example, \cite{AGS, HN,LN}).  
Motivated by this lack of examples, we first construct new explicit four-dimensional Bach-flat manifolds of neutral signature.

\subsection{Riemannian extensions}\label{sse:RE}
In order to introduce the family of metrics under consideration, we recall that a  pseudo-Riemannian manifold $(M,g)$ is a \emph{Walker manifold} if there exists a parallel null distribution $\mathcal{D}$ on $M$. 
Walker metrics, also called Brinkmann waves in the literature, have been widely investigated in the Lorentzian setting (pp-waves being a special class among them). They appear in many geometrical situations showing a specific behaviour without Riemannian counterpart (see \cite{walker}).

Let $(M,g,\mathcal{D})$ be a four-dimensional Walker manifold of neutral signature and
$\mathcal{D}$ of maximal rank. Then there are local
coordinates $(x^1,x^{2}, x_{1^\prime},x_{2^\prime})$ so that the metric $g$ is given by (see \cite{W-50})
\begin{equation}\label{eqn-2.a}
g= 2\, dx^i\circ dx_{i^\prime}+g_{ij}\, dx^i\circ dx^j\,,
\end{equation}
where ``$\circ$" denotes the symmetric
product 
$\omega_1\circ\omega_2:=\textstyle\frac12(\omega_1\otimes\omega_2+\omega_2\otimes\omega_1)$ and
$(g_{ij})$ is a $2\times 2$ symmetric matrix whose entries are functions of all the variables. Moreover, the parallel degenerate distribution is given by
$\mathcal{D}=\operatorname{span}\{\partial_{x_{1^\prime}},\partial_{x_{2^\prime}}\}$.

A special family of four-dimensional Walker metrics is provided by the Riemannian extensions of affine connections to the cotangent bundle of an affine surface. 
Next we briefly sketch their construction.
Let $T^*\Sigma$ be the cotangent bundle of a surface $\Sigma$
and let $\pi\colon T^*\Sigma\rightarrow\Sigma$ be the projection. 
Let $\tilde{p}=(p,\omega)$ denote a point of $T^*\Sigma$, where $p\in \Sigma$ and $\omega\in T_p^*\Sigma$.
Local coordinates $(x^i)$ in an open set $\mathcal{U}$ of $\Sigma$ induce local
coordinates $(x^i,x_{i^\prime})$ in $\pi^{-1}(\mathcal{U})$, where one sets
$\omega=\sum x_{i^\prime}dx^i$.
The evaluation functions on $T^*\Sigma$ play a central role in the construction. They are defined as follows.
For each vector field $X$ on $\Sigma$, \emph{the evaluation of $X$} is the real valued function $\iota
X\colon T^*\Sigma\rightarrow\mathbb{R}$  given by
$\iota X(p,\omega)=\omega(X_p)$. 
Vector fields on $T^*\Sigma$ are characterized by their action on evaluations $\iota X$ and one defines 
the complete lift to $T^*\Sigma$ of a vector field $X$ on $\Sigma$ by $X^C(\iota Z)=\iota [X,Z]$,
for all vector fields $Z$ on $\Sigma$.
Moreover, a $(0,s)$-tensor field on $T^*\Sigma$ is characterized by its action on complete lifts of vector fields on $\Sigma$.

Next, let $\affcon$ be a torsion free affine connection on $\Sigma$. The \emph{Riemannian extension} $g_\affcon$ is
the neutral signature metric $g_\affcon$ on $T^*\Sigma$ characterized by the identity
$g_\affcon(X^C,Y^C)=-\iota(\affcon_XY+\affcon_YX)$ (see \cite{Patterson-Walker52}).
They are expressed in the induced local coordinates $(x^i,x_{i^\prime})$ as follows
\begin{equation}
\label{eq:Riemannian-extension}
g_\affcon= 2\, dx^i\circ dx_{i'}-2x_{k'}{}^\affcon\Gamma_{ij}{}^kdx^i\circ dx^j\,,
\end{equation}
where ${}^\affcon\Gamma_{ij}{}^k$ denote the Christoffel symbols of $\affcon$.
The geometry of $(T^*\Sigma,g_\affcon)$ is strongly related to that of $(\Sigma,\affcon)$. Recall that the curvature of any affine surface is completely determined by its Ricci tensor $\rho^\affcon$. Moreover, the symmetric and skew-symmetric parts 
given by 
$
\rho^\affcon_{sym}(X,Y)=\frac{1}{2}\left\{\rho^\affcon(X,Y)\right.$ $+$ $\left.\rho^\affcon(Y,X)\right\}$, and
$
\rho^\affcon_{sk}(X,Y)=\frac{1}{2}\left\{\rho^\affcon(X,Y)-\rho^\affcon(Y,X)\right\}$
play a distinguished role.

Let $\Phi$ be a symmetric $(0,2)$-tensor field on $\Sigma$. Then the \emph{deformed Riemannian extension}, $g_{\affcon,\Phi}=g_\affcon+\pi^*\Phi$, is a first perturbation of the Riemannian extension. A second one is as follows.
Let $T=T_i^k dx^i\otimes\partial_{x^k}$ be a $(1,1)$-tensor field on $\Sigma$. Its evaluation $\iota T$ defines a one-form on $T^*\Sigma$  characterized by $\iota T(X^C)=\iota(TX)$. 
The \emph{modified Riemannian extension} $g_{\affcon,\Phi,T}$ is the neutral signature metric on
$T^*\Sigma$ defined by (see \cite{CL-GR-G-VL-PRSA-2009})
\begin{equation}\label{eq:Riemannian-extension-2}
g_{\affcon,\Phi,T}=\iota T\circ\iota T + g_\affcon+\pi^*\Phi\,,
\end{equation}
where $\Phi$ is a symmetric $(0,2)$-tensor field on $\Sigma$.
In local coordinates one has
\[
g_{\affcon,\Phi,T}\!=\! 2\, dx^i\circ dx_{i'}\!+\!\{  \frac{1}{2} x_{r'} x_{s'} (T_i^r T_j^s + T_j^r T_i^s)
 \! -\!2x_{k'}{}^\affcon\Gamma_{ij}{}^k\!+\! \Phi_{ij}\} dx^i\circ dx^j.
\]
The case when $T$ is a multiple of the identity ($T=c\operatorname{Id},\,c\neq 0$) is of special interest. It was shown in \cite{CL-GR-G-VL-PRSA-2009} that for any affine surface $(\Sigma,\affcon)$, the modified Riemannian extension $g_{\affcon,\Phi,c\operatorname{Id}}$ is an Einstein metric on $T^*\Sigma$ if and only if the deformation tensor $\Phi$ is the symmetric part of the Ricci tensor of $(\Sigma,\affcon)$.
Moreover, a slight generalization of the modified Riemannian extension allows a complete description of self-dual Walker metrics as follows.

\begin{theorem}{\rm\cite{CL-GR-G-VL-PRSA-2009, DR-GR-VL-06}}\label{thm-7.2}
A four-dimensional Walker metric   is self-dual if and only
if it is locally isometric to the cotangent bundle $T^*\Sigma$ of an affine
surface $(\Sigma,\affcon)$, with metric tensor
$$
g=\iota X(\iota \operatorname{id}\circ\iota \operatorname{id})+ \iota \operatorname{id}\circ\iota
T +g_{\affcon}+\pi^*\Phi\,
$$
where $X$, $T$, $\affcon$ and $\Phi$ are a vector field, a
$(1,1)$-tensor field, a torsion free affine connection and a
symmetric $(0,2)$-tensor field on $\Sigma$, respectively.
\end{theorem}

As a matter of notation, we will write $\partial_k=\frac{\partial}{\partial x^k}$ and 
$\partial_{k'}=\frac{\partial}{\partial x_{k'}}$, unless we want to emphasize some special coordinates. We will let $\phi_k=\frac{\partial}{\partial x^k}\phi$ and 
$\phi_{k'}=\frac{\partial}{\partial x_{k'}}\phi$ to denote the corresponding first derivatives of a smooth function $\phi$.

\subsection{Gradient Ricci solitons and affine gradient Ricci solitons}\label{sse:GRS}
Let $(M,g,f)$ be a gradient Ricci soliton, i.e., $(M,g)$ is a pseudo-Riemannian manifold and $f\in\mathcal{C}^\infty(M)$ is a solution of equation \eqref{eq:ricci-soliton} for some $\lambda\in\mathbb{R}$. The level set hypersurfaces of the potential function play a distinguished role in analyzing the geometry of gradient Ricci solitons. Hence we say that the soliton is \emph{non isotropic} if $\nabla f$ is nowhere lightlike (i.e., $\|\nabla f\|^2\neq 0$), and that the soliton is \emph{isotropic} if $\|\nabla f\|^2=0$, but $\nabla f\neq 0$.

Non isotropic gradient Ricci solitons lead to local warped product decompositions in the locally conformally flat and half conformally flat cases, and their geometry resembles the Riemannian situation \cite{BVGR, BGG}. The isotropic case is, however,  in sharp contrast with the positive definite setting since $\nabla f$ gives rise to a Walker structure. Self-dual gradient Ricci solitons which are not locally conformally flat are isotropic and, moreover,  they are described in terms of  Riemannian extensions as follows.

\begin{theorem}{\rm \cite{BVGR}}\label{th:BVGR1}
Let $(M,g,f)$ be a four-dimensional self-dual gradient Ricci soliton of neutral signature which is not locally conformally flat. Then $(M,g)$ is locally isometric to the cotangent bundle $T^*\Sigma$ of an affine surface $(\Sigma,D)$ equipped with a modified Riemannian extension $g_{D,\Phi,0}$.

Moreover any such gradient Ricci soliton is steady and the potential function is given by $f=h\circ\pi$ for some $h\in\mathcal{C}^\infty(\Sigma)$ satisfying the affine gradient Ricci soliton equation 
\begin{equation}
\label{eq:AGRS}
\operatorname{Hes}^\affcon_h+2\rho^\affcon_{sym}=0\,,
\end{equation}
for any symmetric $(0,2)$-tensor field $\Phi$ on $\Sigma$.
\end{theorem}

The previous result relates the affine geometry of $(\Sigma,\affcon)$ and the pseudo-Riemannian geometry of $(T^*\Sigma,g_{\affcon,\Phi,0})$, allowing the construction of an infinite family of steady gradient Ricci solitons on $T^*\Sigma$ for any initial data $(\Sigma,\affcon, h)$ satisfying \eqref{eq:AGRS}.
It is important to remark here that the existence of affine gradient Ricci solitons imposes some restrictions on $(\Sigma, \affcon)$, as shown in \cite{BVGRG} in the locally homogeneous case.

\section{Bach-flat modified Riemannian extensions}\label{se:Bach-flat}

The use of modified Riemannian extensions with $T=c\operatorname{Id}$ allowed the construction of many  examples of self-dual Einstein metrics \cite{CL-GR-G-VL-PRSA-2009}.
One of the crucial facts in understanding the metrics $g_{\affcon,\Phi,c\operatorname{Id}}$ is that the $(1,1)$-tensor field $T=c\operatorname{Id}$ is parallel with respect to the connection $\affcon$. Hence, a natural generalization arises by considering arbitrary tensor fields $T$ which are parallel with respect to the affine connection $\affcon$. 

Let $(\Sigma,\affcon,T)$ be a torsion free affine surface equipped with a parallel $(1,1)$-tensor field $T$. Parallelizability of $T$ guaranties the existence of local coordinates  $(x^1,x^2)$ on $\Sigma$ so that
\[
	T\partial_1 = T^1_1\, \partial_1 + T^2_1\, \partial_2,
	\qquad
	T\partial_2 = T^1_2\,\partial_1 + T^2_2 \,\partial_2,
\]
for some real constants  $T^j_i$. 
Let $(T^*\Sigma,g_{\affcon,\Phi,T})$ be the modified Riemannian extension given by \eqref{eq:Riemannian-extension-2}. Further note that $\affcon$ and $\Phi$ are taken with full generality. Thus, the corresponding Christoffel symbols ${}^D\Gamma_{ij}^k$  and the coefficient functions $\Phi_{ij}$ are arbitrary smooth functions of the coordinates $(x^1,x^2)$.

Our first main result concerns the construction of Bach-flat metrics:

\begin{theorem}\label{th:Bach-zero}
Let $(\Sigma,\affcon,T)$ be a torsion free affine surface equipped with a parallel $(1,1)$-tensor field $T$. Let $\Phi$ be an arbitrary symmetric $(0,2)$-tensor field on $\Sigma$. Then the Bach tensor of 
$(T^*\Sigma,g_{\affcon,\Phi,T})$ vanishes if and only if  $T$ is either a multiple of the identity or  nilpotent.
\end{theorem}

\begin{proof}
In order to compute the Bach tensor of $(T^*\Sigma,g_{\affcon,\Phi,T})$, first of all observe that being $T$ parallel imposes some restrictions on the components $T_i^j$ as well as on the Christoffel symbols of the connection $\affcon$:
\begin{equation}\label{eq:nablaT}
\affcon T=0:\left\{
\begin{array}{l}
	T^1_2\, {}^\affcon\Gamma_{11}^2 - T^2_1\,{}^\affcon\Gamma_{12}^1=0, 
	
	\\[0.05in]
	
	T^1_2\, {}^\affcon\Gamma_{12}^2 - T^2_1\,{}^\affcon\Gamma_{22}^1=0,
	
	\\[0.05in]
	
	T^2_1\,{}^\affcon\Gamma_{11}^1+(T^2_2-T^1_1)\,{}^\affcon\Gamma_{11}^2-T^2_1\,{}^\affcon\Gamma_{12}^2=0,
	
	\\[0.05in]
	
	T^1_2\,{}^\affcon\Gamma_{11}^1+(T^2_2-T^1_1)\,{}^\affcon\Gamma_{12}^1-T^1_2\,{}^\affcon\Gamma_{12}^2=0,
	
	\\[0.05in]
	
	T^2_1\,{}^\affcon\Gamma_{12}^1+(T^2_2-T^1_1)\,{}^\affcon\Gamma_{12}^2-T^2_1\,{}^\affcon\Gamma_{22}^2=0,
	
	\\[0.05in]
	
	T^1_2\,{}^\affcon\Gamma_{12}^1+(T^2_2-T^1_1)\,{}^\affcon\Gamma_{22}^1-T^1_2\,{}^\affcon\Gamma_{22}^2=0.
\end{array}\right.
\end{equation}

Then, expressing the Bach tensor $\bach_{ij}=\bach(\partial_i,\partial_j)$ in induced coordinates
$(x^i,x_{i'})$, a long but straightforward calculation shows that
\begin{equation}\label{eq:Bach tensor}
(\bach_{ij}) = 
\left(
\begin{array}{c|c}
	\begin{array}{cc}
		\bach_{11} & \bach_{12}
		\\
		\bach_{12} & \bach_{22}
	\end{array}
	&
	\tilde{\bach}
	
	\\
	\hline
	\\[-0.15in]
	
	\tilde{\bach}
	&
	0 
\end{array}
\right),
\end{equation}
where 
\[
\tilde{\bach}=\frac{1}{6} ( (T^1_1 - T^2_2)^2 + 4 T^1_2 T^2_1)\cdot(T^1_1 + T^2_2)\cdot
\left(
\begin{array}{cc}
	T^1_1 - T^2_2 &    2 T^2_1 
	\\
    2 T^1_2  &  T^2_2- T^1_1
\end{array}
\right)
\]
and where the coefficients $\bach_{11}$, $\bach_{12}$ and $\bach_{22}$ can be written in terms of $\dt=\det (T)$ and $\tr=\operatorname{tr}(T)$ as follows:

\medskip

\noindent
$
\begin{array}{l}
\bach_{11}=
-\frac{1}{6}\left\{
10 \dt^3 
- 2  (\tr^2+ 13 T^2_2 \tr-15 (T^2_2)^2  )\dt^2
\right.
\\[0.05in]
\phantom{................}
\left.
+  (5 \tr - T^2_2) ( \tr - T^2_2) \tr^2  \dt
- (\tr - T^2_2)^2 \tr^4 
 \right\}x_{1'}^2
 
\\[0.1in]
\phantom{.........}

-\frac{1}{6}
\left\{
(T^2_1)^2 (30 \dt^2 + \tr^2 \dt - \tr^4)
\right\} x_{2'}^2

\\[0.1in]
\phantom{.........}

- \frac{1}{3} \left\{
 (13\tr - 30 T^2_2) \dt^2  
+  (3\tr - T^2_2) \tr^2 \dt 
-   (\tr - T^2_2) \tr^4
\right\} T^2_1\,  x_{1'} x_{2'}

\\[0.1in]
\phantom{.........}

-\frac{1}{3}\left\{
 ({}^\affcon\Gamma_{11}^1 + 2 {}^\affcon\Gamma_{12}^2)(\tr -2 T^2_2)+ 2 T^2_1 {}^\affcon\Gamma_{22}^2
\right\}  (\tr^2-4 \dt) \tr\,  x_{1'}

\\[0.1in]
\phantom{.........}
- \frac{1}{3}  \left\{
 {}^\affcon\Gamma_{11}^2 (\tr-2 T^2_2)  + 2 T^2_1 {}^\affcon\Gamma_{12}^2 
 \right\}  (\tr^2-4 \dt ) \tr\,  x_{2'}
 
\\[0.1in]
\phantom{.........}

-\frac{1}{6}
\left\{
10 \dt^2 
+ (3 \tr^2- 22 T^2_2 \tr +14 (T^2_2)^2  )\dt 
-  ( \tr^2- 4 T^2_2 \tr + 2 (T^2_2)^2 ) \tr^2
 \right\} \Phi_{11}

\\[0.1in]
\phantom{.........}

-\frac{1}{3} \left\{
( 11 \tr - 14 T^2_2 )\dt 
- 2 (\tr-T^2_2)  \tr^2
\right\}  T^2_1\Phi_{12}

\\[0.1in]
\phantom{.........}

+\frac{1}{3}\left\{
 \tr^2-7 \dt  
\right\}(T^2_1)^2 \Phi_{22}

\\[0.1in]
\phantom{.........}

-\frac{2}{3}
(\partial_2{}^\affcon\Gamma_{11}^2-\partial_1{}^\affcon\Gamma_{12}^2)
(4 \dt - \tr^2) \,,
\end{array}
$

\bigskip

\noindent
$
\begin{array}{l}
\bach_{12}=
-\frac{1}{6}\left\{
 (13 \tr - 30 T^2_2 )  \dt^2
+  ( 3 \tr - T^2_2 )   \tr^2 \dt
- ( \tr-T^2_2 )  \tr^4 
\right\}  T^1_2x_{1'}^2
 
\\[0.1in]
\phantom{.........}

+\frac{1}{6}\left\{
(17 \tr-30 T^2_2) \dt^2
-  (2 \tr + T^2_2)     \tr^2 \dt
+ T^2_2 \tr^4  
\right\} T^2_1 x_{2'}^2

\\[0.1in]
\phantom{.........}

+\frac{1}{6}\left\{
20 \dt^3 
+ 4 (4 \tr^2  - 15 T^2_2 \tr  + 15 (T^2_2)^2) \dt^2
\right.
\\[0.05in]
\phantom{................}
\left. 
-( 3 \tr^2 + 2 T^2_2 \tr - 2 (T^2_2)^2) \tr^2 \dt
+ 2  ( \tr-T^2_2) T^2_2 \tr^4 
\right\} x_{1'} x_{2'}

\\[0.1in]
\phantom{.........}

-\frac{1}{3}\left\{
 {}^\affcon\Gamma_{12}^1(\tr-2 T^2_2 ) 
 + 2  T^2_1  {}^\affcon\Gamma_{22}^1
\right\} ( \tr^2-4 \dt)\tr  x_{1'}

\\[0.1in]
\phantom{.........}

-\frac{1}{3}\left\{
{}^\affcon\Gamma_{12}^2( \tr-2 T^2_2 )  
+   2 T^2_1 {}^\affcon\Gamma_{22}^2
\right\}( \tr^2-4 \dt)\tr   x_{2'}

\\[0.1in]
\phantom{.........}

-\frac{1}{6}\left\{
 ( 11 \tr - 14 T^2_2)  \dt
 - 2  (\tr -T^2_2 ) \tr^2
\right\}  T^1_2 \Phi_{11}

\\[0.1in]
\phantom{.........}

+\frac{1}{6}\left\{
4 \dt^2 
+ (6 \tr^2- 28 T^2_2 \tr + 28 (T^2_2)^2 )\dt
-(\tr-2 T^2_2 )^2   \tr^2 
\right\} \Phi_{12}

\\[0.1in]
\phantom{.........}

+\frac{1}{6}\left\{
( 3 \tr-14 T^2_2 ) \dt
+ 2  T^2_2 \tr^2
\right\}  T^2_1  \Phi_{22}

\\[0.1in]
\phantom{.........}

-\frac{1}{3}\left\{
(\partial_2{}^\affcon\Gamma_{11}^1-\partial_1{}^\affcon\Gamma_{12}^1
-\partial_2{}^\affcon\Gamma_{12}^2+\partial_1{}^\affcon\Gamma_{22}^2)(\tr^2-4\dt)
\right\}  \,,
\end{array}
$

\bigskip

\noindent
$
\begin{array}{l}
\bach_{22}=
-\frac{1}{6}\left\{ 
 30\dt^2 
  - \tr^4 
 +  \tr^2 \dt 
 \right\}  (T^1_2)^2  x_{1'}^2

\\[0.1in]
\phantom{.........}

-\frac{1}{6}\left\{
10 \dt^3 
+  2 ( \tr^2- 17 T^2_2 \tr+ 15 (T^2_2)^2 )\dt^2 
\right.
\\[0.05in]
\phantom{................}
\left. 
+  ( 4 \tr + T^2_2)  T^2_2 \tr^2 \dt
- (T^2_2)^2 \tr^4  
\right\}x_{2'}^2

\\[0.1in]
\phantom{.........}

+\frac{1}{3}\left\{
 ( 17 \tr-30 T^2_2)\dt^2
-  ( 2 \tr  +T^2_2) \tr^2 \dt
+ T^2_2 \tr^4 
\right\}  T^1_2  x_{1'} x_{2'}

\\[0.1in]
\phantom{.........}

-\frac{1}{3}\left\{
    {}^\affcon\Gamma_{22}^1( \tr-2 T^2_2) + 
   2 T^1_2 {}^\affcon\Gamma_{22}^2\right\} ( \tr^2-4 \dt) \tr x_{1'}
   
\\[0.1in]
\phantom{.........}

+\frac{1}{3}\left\{
{}^\affcon\Gamma_{22}^2(\tr-2 T^2_2 )
-2 T^2_1 {}^\affcon\Gamma_{22}^1 
\right\}( \tr^2-4 \dt) \tr  x_{2'}

\\[0.1in]
\phantom{.........}

-\frac{1}{3}
  (7 \dt - \tr^2) (T^1_2)^2 
\Phi_{11}

\\[0.1in]
\phantom{.........}

+\frac{1}{3}\left\{
( 3 \tr-14 T^2_2) T^1_2 \dt
+2 T^1_2 T^2_2 \tr^2 
\right\}\Phi_{12}

\\[0.1in]
\phantom{.........}

-\frac{1}{6}\left\{
10 \dt^2 
- ( 5 \tr^2 + 6 T^2_2 \tr-14 (T^2_2)^2) \dt
+ \tr^4 
- 2 (T^2_2)^2 \tr^2
\right\}\Phi_{22}

\\[0.1in]
\phantom{.........}

-\frac{2}{3}
(\partial_2{}^\affcon\Gamma_{12}^1-\partial_1{}^\affcon\Gamma_{22}^1)
(\tr^2-4\dt)  \,.
\end{array}
$

\bigskip

Suppose  first that  the Bach tensor of $(T^*\Sigma,g_{\affcon,\Phi,T})$ vanishes. We start analyzing   the case $T^1_2=0$. In this case, the expression of $\tilde \bach$ in  Equation~\eqref{eq:Bach tensor} reduces to
\begin{equation}\label{eq:Btilde P12=0}
\tilde{\bach}=\frac{1}{6}  (T^1_1 - T^2_2)^2\cdot(T^1_1 + T^2_2)\cdot
\left(
\begin{array}{cc}
	T^1_1 - T^2_2 &  2  T^2_1 
	\\
    0  & T^2_2 - T^1_1
\end{array}
\right).
\end{equation}
If $T^2_2 = T^1_1$, we differentiate the component $\bach_{11}$ in  Equation~\eqref{eq:Bach tensor}  twice with respect to $x_{2'}$  to obtain  $T^2_1 T^1_1=0$. Thus, either $T^2_1=0$ and $T$ is a multiple of the identity, or $T^1_1=0$ and, in such a case, $T$ is determined by $T\partial_1=T^2_1\partial_2$ and therefore it is nilpotent. If $T^2_2\neq T^1_1$, then Equation~\eqref{eq:Btilde P12=0} implies that $T^2_2=-T^1_1$. In this case, we differentiate the component $\bach_{22}$ in Equation~\eqref{eq:Bach tensor} twice with respect to $x_{2'}$ and obtain $T^1_1=0$. Thus, as before, $T$ is nilpotent.

\medskip

Next we analyze the case $T^1_2\neq 0$. We use Equation~\eqref{eq:nablaT} to express
\[
\begin{array}{ll}
	{}^\affcon\Gamma_{11}^1=\tfrac{T^1_1-T^2_2}{T^1_2}{}^\affcon\Gamma_{12}^1
		       +\tfrac{T^2_1}{T^1_2} {}^\affcon\Gamma_{22}^1,
	&\quad
	{}^\affcon\Gamma_{11}^2=\tfrac{T^2_1}{T^1_2} {}^\affcon\Gamma_{12}^1,
	\\[0.1in]
	{}^\affcon\Gamma_{12}^2=\tfrac{T^2_1}{T^1_2} {}^\affcon\Gamma_{22}^1,
	&\quad
	{}^\affcon\Gamma_{22}^2={}^\affcon\Gamma_{12}^1
	        -\tfrac{T^1_1-T^2_2}{T^1_2}{}^\affcon\Gamma_{22}^1.
\end{array}
\]
Considering the component  $\tilde \bach_{11}$ in Equation~\eqref{eq:Bach tensor},
\[
	\tilde \bach_{11}=\frac{1}{6} (T^1_1 - T^2_2)\cdot(T^1_1 + T^2_2) \cdot
	 ( (T^1_1 - T^2_2)^2 + 4 T^1_2 T^2_1),
\]
we  analyze separately  the vanishing of each one of the three factors in  $\tilde \bach_{11}$.

Assume that $T^2_2=T^1_1$. In this case,  component $\tilde \bach_{12}$ in  Equation~\eqref{eq:Bach tensor} reduces to $\tilde \bach_{12}=\frac{8}{3}T^1_2 (T^2_1)^2 T^1_1$; since we are assuming that  $T^1_2\neq 0$, then either $T^2_1=0$ or $T^2_1\neq 0$ and $T^1_1=0$. If $T^2_1=0$,   the only non-zero component of the Bach tensor is given by $\bach_{22}=-(T^1_2)^2(T^1_1)^2(3(T^1_1)^2 x_{1'}^2+\Phi_{11})$, from where it follows that $T^1_1=0$ and hence $T$ is determined by $T\partial_2=T^1_2\partial_1$ and is nilpotent. If $T^2_1\neq 0$ and $T^1_1=0$, then we differentiate the component $\bach_{12}$ in  Equation~\eqref{eq:Bach tensor} with respect to $x_{1'}$ and $x_{2'}$ to get $T^1_2 T^2_1=0$, which is not possible since both $T^1_2$ and  $T^2_1$ are non-null.

Suppose now that $T^2_2=-T^1_1$. In this case, we differentiate the component $\bach_{22}$ in  Equation~\eqref{eq:Bach tensor} twice with respect to $x_{1'}$ and as a consequence we obtain $T^1_2 (T^1_2 T^2_1+(T^1_1)^2)=0$; since we are assuming $T^1_2\neq 0$, it follows that $T^2_1=-\frac{(T^1_1)^2}{T^1_2}$. Thus, the (1,1)-tensor field  $T$ is given by $T\partial_1 = T^1_1\partial_1-\frac{(T^1_1)^2}{T^1_2}\partial_2$ and $T\partial_2=T^1_2\partial_1-T^1_1\partial_2$, and therefore it is nilpotent as well.

Finally, suppose that $	  (T^1_1 - T^2_2)^2 + 4 T^1_2 T^2_1=0$; since $T^1_2\neq 0$, this is equivalent to $T^2_1=-\frac{(T^1_1-T^2_2)^2}{4T^1_2}$. Now, we differentiate the component $\bach_{22}$ in  Equation~\eqref{eq:Bach tensor} twice with respect to $x_{1'}$ to obtain $T^1_2(T^1_1+T^2_2)=0$. Thus, we have that $T^2_2=-T^1_1$ and $T$ is given by $T\partial_1 = T^1_1\partial_1-\frac{(T^1_1)^2}{ T^1_2}\partial_2$ and $T\partial_2=T^1_2\partial_1-T^1_1\partial_2$, which again implies that $T$ is nilpotent.

\medskip

To conclude the proof we show the ``only if'' part. If $T$ is a multiple of the identity, then $(T^*\Sigma,g_{\affcon,\Phi,T})$ is self-dual by Theorem \ref{thm-7.2} and therefore it has vanishing Bach tensor. Thus, we suppose $T$ is parallel and  nilpotent and,  in this case, we can choose a system of coordinates $(x^1,x^2)$ such that $T$ is determined  by $T\partial_1=\partial_2$ and $T\partial_2=0$. Hence, examining Equation~\eqref{eq:Bach tensor}, clearly  $\tilde \bach=0$ and,  since $\dt=\tr=0$, one easily checks that $\bach_{11}=\bach_{12}= \bach_{22}=0$, showing that  the Bach tensor of $(T^*\Sigma,g_{\affcon,\Phi,T})$ vanishes.      
\end{proof}

\begin{remark}\label{re:3-2}
\rm
We emphasize that even though the Bach tensor of the metrics $g_{\affcon,\Phi,T}$ depends on the choice of $\Phi$ (as shown in the proof of Theorem~\ref{th:Bach-zero}), the existence of Bach-flat metrics in Theorem \ref{th:Bach-zero} is independent of the symmetric $(0,2)$-tensor field $\Phi$, thus providing an infinite family of examples for each initial data $(\Sigma,\affcon, T)$.
Moreover, note that the metrics $g_{\affcon,\Phi,T}$ are generically non-isometric for different deformation tensor fields $\Phi$.
\end{remark}

The Bach-flat modified Riemannian extensions in Theorem \ref{th:Bach-zero} obtained from a $(1,1)$-tensor field of the form $T=c\operatorname{Id}$ are not of interest for our purposes since they all are half conformally flat (cf. Theorem \ref{thm-7.2}). Hence, in what follows we focus on the case when $T$ is a parallel nilpotent $(1,1)$-tensor field and refer to $g_{\affcon,\Phi,T}$ as a \emph{nilpotent Riemannian extension}.

\subsection{Affine connections supporting parallel nilpotent tensors}

The proof of Theorem \ref{th:Bach-zero} shows that the existence of a parallel 
nilpotent tensor field $T$ on a torsion free affine surface $(\Sigma,\affcon)$ imposes some restrictions 
on~$\affcon$.

\begin{proposition}\label{prop:1}
Let $(\Sigma,\affcon,T)$ be a torsion free affine surface equipped with a nilpotent $(1,1)$-tensor field $T$. If $T$ is parallel, then
\begin{enumerate}
\item[(i)] $\operatorname{ker}T$ is a parallel one-dimensional distribution whose integral curves are geodesics of $(\Sigma,\affcon)$.
\item[(ii)] The symmetric part of the Ricci tensor, $\rho^\affcon_{sym}$, is zero or of rank one
and recurrent, i.e.,  $\affcon\rho^\affcon_{sym}=\eta\otimes\rho^\affcon_{sym}$, for some one-form $\eta$.
\end{enumerate}
\end{proposition}

\begin{proof}
Let $(\Sigma,D)$ be a torsion free affine surface admitting a parallel nilpotent $(1,1)$-tensor field $T$. Then there exist suitable coordinates $(x^1,x^2)$ where $T\partial_1=\partial_2$, $T\partial_2=0$ and it follows from \eqref{eq:nablaT} that the Christoffel symbols of $\affcon$ satisfy
\begin{equation}\label{eq:Christ-T-parallel}
    {}^\affcon\Gamma_{12}^1=0,
	\qquad
	{}^\affcon\Gamma_{12}^2={}^\affcon\Gamma_{11}^1,
	\qquad
	{}^\affcon\Gamma_{22}^1=0,
	\qquad
{}^\affcon\Gamma_{22}^2=0\,.
\end{equation}
In such a case the one-dimensional distribution $\operatorname{ker}T(=\operatorname{span}\{\partial_2\})$ is parallel and $\partial_2$ is a geodesic vector field, thus showing Assertion (i). Moreover, the Ricci tensor of any affine connection given by \eqref{eq:Christ-T-parallel} satisfies
\[
\rho^\affcon = \left(\begin{array}{cc}
\partial_2{}^\affcon\Gamma_{11}^2-\partial_1{}^\affcon\Gamma_{11}^1 &\partial_2{}^\affcon\Gamma_{11}^1\\
\noalign{\medskip}
-\partial_2{}^\affcon\Gamma_{11}^1&0
\end{array}\right)\,,
\]
from where it follows that the symmetric and the skew-symmetric parts of the Ricci tensor are given by
\[
\rho^\affcon_{sym}= (\partial_2{}^\affcon\Gamma_{11}^2-\partial_1{}^\affcon\Gamma_{11}^1)\, dx^1\circ dx^1  \,,
\qquad 
\rho^\affcon_{sk}= \partial_2{}^\affcon\Gamma_{11}^1\, dx^1\wedge dx^2\,.
\]
Hence $\rho^\affcon_{sym}$ is either zero or of rank one.
Moreover, a straightforward calculation of the covariant derivative of the symmetric part of the Ricci tensor gives
$$
\begin{array}{rcl}
(\affcon_{\partial_1}\rho^\affcon_{sym})(\partial_1,\partial_1)&=&
\partial_{12}{}^\affcon\Gamma_{11}^2-\partial_{11}{}^\affcon\Gamma_{11}^1
-2{}^\affcon\Gamma_{11}^1(\partial_2{}^\affcon\Gamma_{11}^2-\partial_1{}^\affcon\Gamma_{11}^1)\,,
\\
\noalign{\medskip}
(\affcon_{\partial_2}\rho^\affcon_{sym})(\partial_1,\partial_1)&=&
\partial_{22}{}^\affcon\Gamma_{11}^2-\partial_{12}{}^\affcon\Gamma_{11}^1\,,
\end{array}
$$
the other components being zero. This shows that $\rho^\affcon_{sym}$ is recurrent, i.e.,  $\affcon\rho^\affcon_{sym}=\eta\otimes\rho^\affcon_{sym}$, with recurrence one-form
\begin{equation}
\label{eq:recurrence-eta}
\eta=\{\partial_1 \ln\rho^\affcon_{sym}(\partial_1,\partial_1)-2{}^\affcon\Gamma_{11}^1 \}\, dx^1
+ \partial_2\ln\rho^\affcon_{sym}(\partial_1,\partial_1)\, dx^2\,,
\end{equation}
which proves (ii).
\end{proof}

\begin{remark}
\label{re:recurrente}\rm
It follows from the expression of $\rho^\affcon_{sk}$ in the proof of Proposition \ref{prop:1} that any connection given by \eqref{eq:Christ-T-parallel} has symmetric Ricci tensor if and only if $\partial_2{}^\affcon\Gamma_{11}^1=0$, in which case $\rho^\affcon$ is recurrent.
Now, it follows from the work of Wong \cite{Wong} that any such connection can be described in suitable coordinates $(\bar u^1,\bar u^2)$ by $\affcon{\partial_{\bar u^1}}\partial_{\bar u^1}={}^{\bar u}\Gamma_{11}^2(\bar u^1,\bar u^2) \partial_{\bar u^2}$, where ${}^{\bar u}\Gamma_{11}^2(\bar u^1,\bar u^2)$ is an arbitrary function satisfying $\partial_{\bar u^2}{}^{\bar u}\Gamma_{11}^2(\bar u^1,\bar u^2)\neq 0$. Moreover, the only non-zero component of the Ricci tensor is $\rho^\affcon(\partial_{\bar u^1},\partial_{\bar u^1})=\partial_{\bar u^2}{}^{\bar u}\Gamma_{11}^2$, and  the recurrence one-form $\omega$ is given by
\begin{equation}
\label{eq:recurrence-ricci}
\omega=\partial_{\bar u^1}(\ln\partial_{\bar u^2}{}^{\bar u}\Gamma_{11}^2) d\bar u^1
+ \partial_{\bar u^2}(\ln\partial_{\bar u^2}{}^{\bar u}\Gamma_{11}^2) d\bar u^2.
\end{equation}

Further assume that $T$ is a parallel nilpotent $(1,1)$-tensor field on $(\Sigma,\affcon)$. Then a straightforward calculation shows that its expression in the coordinates $(\bar u^1,\bar u^2)$ is given by 
$T\partial_{\bar u^1}=T_1^2\, \partial_{\bar u^2}$ and $T\partial_{\bar u^2}=0$, for some   $T_1^2\in\mathbb{R}$, $T_1^2\neq 0$. Hence, considering the modified coordinates $(u^1,u^2)=(\bar u^1,(T_1^2)^{-1}\bar u^2)$ one has that 
$T\partial_{u^1}=\partial_{u^2}$ and $T\partial_{u^2}=0$, and the connection is determined by the only non-zero Christoffel symbol ${}^u\Gamma_{11}^2$.
Moreover it follows from the expression of the recurrence one-form $\omega$ that
$\omega(\operatorname{ker}T)=0$ if and only if $\partial_{22}{}^u\Gamma_{11}^2=0$.
\end{remark}

\section{Bach-flat gradient Ricci solitons}\label{se:4}

Let $\Phi$ be a symmetric $(0,2)$-tensor field on $(\Sigma,\affcon,T)$. One  uses the nilpotent structure $T$ to construct an associated symmetric $(0,2)$-tensor field $\widehat{\Phi}$ given by $\widehat{\Phi}(X,Y)$ $=$ $\Phi(TX,TY)$, for all vector fields $X,Y$ on $\Sigma$. 
Further, let $(x^1,x^2)$ be local coordinates where $T\partial_1=\partial_2$, $T\partial_2=0$ and let $\Phi=\Phi_{ij}dx^i\otimes dx^j$. Then $\widehat\Phi$ expresses as
$\widehat{\Phi}=\widehat{\Phi}_{ij}dx^i\otimes dx^j=\Phi_{22}dx^1\otimes dx^1$.

\subsection{Einstein nilpotent Riemannian extensions}

\begin{theorem}\label{re:Ricci}
Let $(\Sigma,\affcon,T)$ be an affine surface equipped with a parallel nilpotent $(1,1)$-tensor field $T$ and let $\Phi$ be a symmetric $(0,2)$-tensor field on $\Sigma$. Then $(T^*\Sigma,g_{\affcon,\Phi,T})$ is Einstein (indeed, Ricci-flat) if and only if $\widehat{\Phi}=-2\rho^\affcon_{sym}$. 
\end{theorem}

\begin{proof}
Let $(x^1,x^2)$ be local coordinates on $\Sigma$ so that $T\partial_1=\partial_2$, $T\partial_2=0$, and consider the induced coordinates $(x^1,x^2,x_{1'},x_{2'})$ on $T^*\Sigma$.
A straightforward calculation shows that the Ricci tensor of any nilpotent Riemannnian extension $g_{\affcon,\Phi,T}$ is determined by
\[
\rho(\partial_{1},\partial_{1})=\Phi(\partial_{2},\partial_{2})+2\rho^\affcon_{sym}(\partial_{1},\partial_{1})\,,
\]
the other components being zero. Hence the Ricci operator is nilpotent and $g_{\affcon,\Phi,T}$ has zero scalar curvature. 
Moreover, the Ricci tensor vanishes if and only if $\Phi(\partial_{2},\partial_{2})+2\rho^\affcon_{sym}(\partial_{1},\partial_{1})=0$. The result now follows. 
\end{proof}

\begin{remark}\rm
\label{prop:Cotton}
The Weyl tensor of a pseudo-Riemannian manifold is harmonic if and only if the Cotton tensor vanishes.
Let $(\Sigma,\affcon,T)$ be an affine surface equipped with a parallel nilpotent $(1,1)$-tensor field $T$ and let $\Phi$ be a symmetric $(0,2)$-tensor field on $\Sigma$. Let $(x^1,x^2)$ be local coordinates on $\Sigma$ so that $T\partial_1=\partial_2$, $T\partial_2=0$, and consider the induced coordinates $(x^1,x^2,x_{1'},x_{2'})$ on $T^*\Sigma$. A straightforward calculation shows that the Cotton tensor of $(T^*\Sigma,g_{\affcon,\Phi,T})$ is given by
$$
\mathfrak{C}(\partial_1,\partial_2,\partial_1)=-\{\partial_2\,\Phi(\partial_2,\partial_2)+2\partial_2\,\rho^\affcon_{sym}(\partial_1,\partial_1) \}\,,
$$
the other components being zero. 
Hence
\emph{$(T^*\Sigma,g_{\affcon,\Phi,T})$ has harmonic Weyl tensor if and only if $\widehat{\affcon\Phi}=-2\,\widehat{\eta}\otimes\rho^\affcon_{sym}$, where
	$\widehat{\eta}(X)=\eta(TX)$, $\eta$ being the recurrence one-form in \eqref{eq:recurrence-eta}, and  $\widehat{\affcon\Phi}(X,Y;Z)=\affcon\Phi(TX,TY;TZ)$.}
\end{remark}

\subsection{Gradient Ricci solitons on nilpotent Riemannian extensions}

Recall from Theorem \ref{th:BVGR1} that the affine gradient Ricci soliton equation $\operatorname{Hes}^\affcon_h+2\rho^\affcon_{sym}=0$ determines the potential function of any self-dual gradient Ricci soliton which is not locally conformally flat, independently of the deformation tensor $\Phi$. 
The next theorem shows that, in contrast with the previous situation, for any $h\in\mathcal{C}^\infty(\Sigma)$ with $dh(\operatorname{ker}T)=0$, one may use the symmetric $(0,2)$-tensor field $\operatorname{Hes}^\affcon_h+2\rho^\affcon_{sym}$ to determine a deformation tensor field $\Phi$  so that the resulting nilpotent Riemannian extension is a Bach-flat steady gradient Ricci soliton with potential function $f=h\circ\pi$.

Let  $(T^*\Sigma,g_{\affcon,\Phi,T},f)$ be a gradient Ricci soliton with potential function 
$f\in\mathcal{C}^\infty(T^*\Sigma)$.
Let $(x^1,x^2)$ be local coordinates on $\Sigma$ so that $T\partial_1=\partial_2$, $T\partial_2=0$, and consider the induced coordinates $(x^1,x^2,x_{1'},x_{2'})$ on $T^*\Sigma$.
Since $\displaystyle\operatorname{Hes}_f(\partial_{{i'}},\partial_{{j'}})
=\partial_{{i'}{j'}} f(x^1,x^2, x_{1'},x_{2'})$, it follows from the expression of the Ricci tensor in Theorem \ref{re:Ricci} and the metric tensor \eqref{eq:Riemannian-extension-2} that 
$f=\iota X+h\circ\pi$ for some $h\in\mathcal{C}^\infty(\Sigma)$ and some vector field $X$ on $\Sigma$.
Set $X=A(x^1,x^2)\partial_1+B(x^1,x^2)\partial_2$ in the local coordinates $(x^1,x^2)$, for some $A,B\in\mathcal{C}^\infty(\Sigma)$. Then
$\displaystyle\operatorname{Hes}_f(\partial_{{2}},\partial_{{1'}})=\partial_2A(x^1,x^2)$, from where it follows that 
$X=A(x^1)\partial_1+B(x^1,x^2)\partial_2$. Considering the component
$\displaystyle\operatorname{Hes}_f(\partial_{{2}},\partial_{{2'}})=-A''(x^1)+\partial_2B(x^1,x^2)$,
one has that 
$X=A(x^1)\partial_1+(P(x^1)+x^2 A'(x^1))\partial_2$ for some smooth function $P(x^1)$.
Next the component
$$
\begin{array}{rcl}
\displaystyle\operatorname{Hes}_f(\partial_{{1}},\partial_{{2'}})
&=&A(x^1){}^\affcon\Gamma_{11}^2-x_{2'}A(x^1)\\
\noalign{\medskip}
&&+{}^\affcon\Gamma_{11}^1(P(x^1)+x^2A'(x^1))+P'(x^1)+x^2A''(x^1)\,,
\end{array}
$$
shows that $A=0$ and it reduces to
$\displaystyle\operatorname{Hes}_f(\partial_{{1}},\partial_{{2'}})=P'(x^1)+P(x^1){}^\affcon\Gamma_{11}^1$. A solution $P(x^1)$ of the equation $P'(x^1)+P(x^1){}^\affcon\Gamma_{11}^1=0$ either vanishes identically (and hence $X=0$) or it is nowhere zero, in which case $\partial_2{}^\affcon\Gamma_{11}^1=0$
(see the proof of Theorem \ref{th:conformally-Einstein}). In the later case Proposition \ref{prop:1} shows that  the Ricci tensor of $(\Sigma,\affcon)$ is symmetric and thus recurrent of rank one.

Since we are mainly interested in the case when $\rho^\affcon_{sk}$ is non-zero, the next theorem examines the simpler situation when $X=0$ and $f=h\circ\pi$.

\begin{theorem}
\label{th:solitons}
Let $(\Sigma,\affcon,T)$ be an affine surface equipped with a parallel nilpotent $(1,1)$-tensor field $T$ and let $\Phi$ be a symmetric $(0,2)$-tensor field on $\Sigma$. 
Let $h\in\mathcal{C}^\infty(\Sigma)$ be a smooth function. Then  $(T^*\Sigma,g_{\affcon,\Phi,T},f=h\circ\pi)$ is a Bach-flat gradient Ricci soliton if and only if $dh(\operatorname{ker}T)=0$ and
\begin{equation}
\label{eq:th-solitons-1}
\widehat{\Phi}=-\operatorname{Hes}^\affcon_h-2\rho^\affcon_{sym}\,.
\end{equation}
Moreover the soliton is steady and isotropic.
\end{theorem}

\begin{proof}
Taking local coordinates on $T^*\Sigma$ as above and setting $f=h\circ\pi$, one has that
$
\operatorname{Hes}_f(\partial_{1},\partial_{{1'}})+\rho(\partial_{1},\partial_{{1'}})=\lambda g(\partial_{1},\partial_{{1'}})
$
leads to $\lambda=0$, which shows that the soliton is steady.
A straightforward calculation shows that the remaining non-zero terms in the gradient Ricci soliton equation are given by
$$
\begin{array}{rcl}
\operatorname{Hes}_f(\partial_{2},\partial_{2})+\rho(\partial_{2},\partial_{2})&=&\partial_{22}h\,,
\\
\noalign{\medskip}
\operatorname{Hes}_f(\partial_{1},\partial_{{2}})+\rho(\partial_{1},\partial_{{2}})&=&\partial_{12}h -{}^\affcon\Gamma_{11}^1\partial_{2}h
\,,
\\
\noalign{\medskip}
\operatorname{Hes}_f(\partial_{1},\partial_{{1}})+\rho(\partial_{1},\partial_{{1}})&=&
x_{2'}\,\partial_{2}h-{}^\affcon\Gamma_{11}^2\,\partial_{2}h+\partial_{11}h-{}^\affcon\Gamma_{11}^1\partial_{1}h
\\
\noalign{\medskip}
&&+\Phi_{22}+2\partial_{2}{}^\affcon\Gamma_{11}^2
-2\partial_{1}{}^\affcon\Gamma_{11}^1
\,.
\end{array}
$$
It immediately follows from the equation $(\operatorname{Hes}_f+\rho)(\partial_{1},\partial_{{1}})=0$ that $\partial_2h=0$, which shows that $dh(\operatorname{ker}T)=0$.  
The only remaining equation now becomes
$$
\begin{array}{rcl}
\operatorname{Hes}_f(\partial_{1},\partial_{{1}})+\rho(\partial_{1},\partial_{{1}})&=&
\partial_{11}h-{}^\affcon\Gamma_{11}^1\partial_{1}h
+\Phi_{22}+2\partial_{2}{}^\affcon\Gamma_{11}^2
-2\partial_{1}{}^\affcon\Gamma_{11}^1
\\

\noalign{\medskip}
&=&\Phi(\partial_2,\partial_2)+\operatorname{Hes}^\affcon_h(\partial_{1},\partial_{1})
+2\rho^\affcon_{sym}(\partial_{1},\partial_{1})\,,
\end{array}
$$
from which Equation \eqref{eq:th-solitons-1} follows.
Moreover, it also follows from the form of the potential function that $\nabla f=h'(x^1)\partial_{{1'}}$, and thus $\|\nabla f\|^2=0$ (equivalently the level hypersurfaces of the potential function are degenerate submanifolds of $T^*\Sigma$), which shows that the soliton is isotropic.
\end{proof}

\begin{remark}
\rm\label{re:4-3}
The tensor field $\mathbb{D}_{ijk}= \mathfrak{C}_{ijk}+W_{ijk\ell}\nabla_\ell f$ introduced in  \cite{GRS2} plays an essential role in analyzing the geometry of Bach-flat gradient Ricci solitons. Local conformal flatness in \cite{GRS1,GRS2} follows from $\mathbb{D}=0$, which is obtained under some natural assumptions.

Gradient Ricci solitons in Theorem \ref{th:solitons} satisfy $\nabla f=h'(x^1)\partial_{{1'}}$. Then, a straightforward calculation shows that $\mathbb{D}$ is completely determined by 
$\mathbb{D}_{121}=-2h'(x^1)\partial_2\,{}^\affcon\Gamma_{11}^1(x^1,x^2)$,
the other components being zero. Hence it follows from the proof of Proposition \ref{prop:1} that the tensor field $\mathbb{D}$ vanishes if and only if the Ricci tensor $\rho^\affcon$ is symmetric. 
However Theorem \ref{th:sd-asd} shows that $(T^*\Sigma,g_{\affcon,\Phi,T})$ is never locally conformally flat. Further observe that, since the soliton is isotropic,  $\nabla f$ does not give rise to a local warped product decomposition unlike the Riemannian case.
\end{remark}

A special situation of Theorem \ref{th:solitons} occurs if $h\in\mathcal{C}^\infty(\Sigma)$ is a solution of the affine gradient Ricci soliton equation \eqref{eq:AGRS}.
The following result will be used in Section \ref{asd-grs} to construct examples of anti-self-dual steady gradient Ricci solitons (cf. Theorem \ref{prop:11}).

\begin{proposition}
\label{prop:2}
Let $(\Sigma,\affcon,T)$ be an affine surface equipped with a parallel nilpotent $(1,1)$-tensor field $T$. Then $(\Sigma,\affcon,T,h)$ is an affine gradient Ricci soliton with $dh(\operatorname{ker}T)=0$ if and only if $(T^*\Sigma,g_{\affcon,\widehat\Phi,T},f=h\circ\pi)$ is a Bach-flat gradient Ricci soliton for any symmetric $(0,2)$-tensor field $\Phi$.

Moreover, $(\Sigma,\affcon,T,h)$ is a non-flat affine gradient Ricci soliton with potential function $h$ satisfying $dh(\operatorname{ker}T)=0$ if and only if the recurrence one-form $\eta$ in \eqref{eq:recurrence-eta} satisfies $\widehat \eta=0$.
\end{proposition}

\begin{proof}
Let $\Phi$ be an arbitrary symmetric $(0,2)$-tensor field on $\Sigma$ and let $\widehat{\Phi}(X,Y)=\Phi(TX,TY)$ be the associated tensor field induced by $T$. Since $\widehat{\Phi}_{22}=0$, \eqref{eq:th-solitons-1} shows that $(T^*\Sigma,g_{\affcon,\widehat\Phi,T},f=h\circ\pi)$ is a gradient Ricci soliton if and only if $(\Sigma,\affcon,T,h)$ is an affine gradient Ricci soliton with $dh(\operatorname{ker}T)=0$.

Next take local coordinates $(x^1,x^2)$ on $\Sigma$  so that $T\partial_1=\partial_2$, $T\partial_2=0$. 
Since the Christoffel symbols ${}^\affcon\Gamma_{ij}^k$  are given by \eqref{eq:Christ-T-parallel}, using the expression of $\rho^\affcon_{sym}$ in Proposition \ref{prop:1}, one has  $(\operatorname{Hes}^\affcon_h+2\rho^\affcon_{sym})(\partial_2,\partial_2)=\partial_{22}h$. Thus $h(x^1,x^2)=x^2 P(x^1)+Q(x^1)$ for some $P,Q\in\mathcal{C}^\infty(\Sigma)$. 
Hence $dh(\operatorname{ker}T)=0$ holds if and only if $P=0$.
Since $h(x^1,x^2)=Q(x^1)$ one has that  
$(\operatorname{Hes}^\affcon_h+2\rho^\affcon_{sym})(\partial_1,\partial_2)=0$, 
and the only remaining equation is 
$$ 
0=(\operatorname{Hes}^\affcon_h+2\rho^\affcon_{sym})(\partial_1,\partial_1)=Q''+2(\partial_2 {}^\affcon\Gamma_{11}^2-\partial_1{}^\affcon\Gamma_{11}^1)=Q''+2\rho^\affcon(\partial_1,\partial_1)\,.
$$
Therefore, the integrability condition becomes $\partial_2\rho^\affcon(\partial_1,\partial_1)=0$.
Hence, it follows from \eqref{eq:recurrence-eta} that $(\Sigma,\affcon,T,h)$ is an affine gradient Ricci soliton with $dh(\operatorname{ker}T)=0$ if and only if the symmetric part of the Ricci tensor $\rho^\affcon_{sym}$ is recurrent with recurrence one-form $\eta$ satisfying $\eta(\operatorname{ker}T)=0$.
\end{proof}

\section{Half conformally flat nilpotent Riemannian extensions}\label{se:W-half}

The existence of a null distribution $\mathcal{D}$ on a four-dimensional manifold $(M,g)$ of neutral signature defines a natural orientation on $M$: the one which, for any basis $u, v$ of $\mathcal{D}$, makes
the bivector $u\wedge v$ self-dual (see \cite{rm08}).
We consider on $T^*\Sigma$ the orientation which agrees with $\mathcal{D}=\operatorname{ker}\pi_*$, and thus self-duality and anti-self-duality are not interchangeable.
The following result shows that they are essentially different for nilpotent Riemannian extensions.

\begin{theorem}\label{th:sd-asd}
Let $(\Sigma,\affcon,T)$ be an affine surface equipped with a parallel nilpotent $(1,1)$-tensor field $T$. Then  
\begin{enumerate}
\item[(i)] $(T^*\Sigma,g_{\affcon,\Phi,T})$ is never self-dual for any deformation tensor field $\Phi$. 
\item[(ii)] If $(T^*\Sigma,g_{\affcon,\Phi,T})$ is anti-self-dual, then $D$ is either a flat connection or $(\Sigma,D)$ is recurrent with symmetric Ricci tensor of rank one. 

In the later case there exist local  coordinates  $(u^1,u^2)$ where the only non-zero Christoffel symbol is ${}^u\Gamma_{11}^2$ and the tensor field $T$ is given by $T\partial_{u^1}=\partial_{u^2}$, $T\partial_{u^2}=0$. 
Moreover, $(T^*\Sigma,g_{\affcon,\Phi,T})$ is anti-self-dual if and only if the symmetric $(0,2)$-tensor field $\Phi$ satisfies the equations:
\begin{equation}\label{eq:asd-pde}
\begin{array}{rcl}
\widehat{\affcon\Phi}&=&-2\widehat{\omega}\otimes\rho^\affcon\,,
\\
\noalign{\medskip}
0&=&
\frac{1}{2}\widehat{\Phi}\otimes\widehat{\Phi}(\partial_1,\partial_1,\partial_1,\partial_1)+2(\widehat{\Phi}\otimes\rho^\affcon)(\partial_1,\partial_1,\partial_1,\partial_1)\\
\noalign{\medskip}
&&
+\affcon^2\Phi(\partial_1,\partial_1; T\partial_1,T\partial_1)
+\affcon^2\Phi(T\partial_1,T\partial_1; \partial_1,\partial_1)\\
\noalign{\medskip}
&&
-2\,\affcon^2\Phi(\partial_1,T\partial_1; T\partial_1,\partial_1)\,,
\end{array}
\end{equation}
where $\widehat{\affcon\Phi}(X,Y,Z)=\affcon\Phi(TX,TY;TZ)$, $\omega$ is the recurrence one-form given by $\affcon\rho^\affcon=\omega\otimes\rho^\affcon$, and $\widehat{\omega}(X)$ $=$ $\omega(TX)$.
\end{enumerate}
\end{theorem}

\begin{proof}
A direct computation using the expression of the anti-self-dual curvature operator of any four-dimensional Walker metric obtained in \cite{DR-GR-VL-06} shows that, for any nilpotent Riemannian extension $g_{\affcon,\Phi,T}$, $W^-$  takes the form
\begin{equation}\label{Weyl-self-dual-nilpotent}
	W^- =\frac{1}{2}
	\left(
	\begin{array}{rrr}
		-1 & 0 & 1
		\\
		0 & 0 & 0
		\\
		-1 & 0 & 1
	\end{array}
	\right)\,,
\end{equation}
thus showing that the anti-self-dual Weyl curvature opertor $W^-$ is nilpotent and hence $(T^*\Sigma,g_{\affcon,\Phi,T})$ is never self-dual, which proves (i).

Next we show (ii). 
Let $(M,g)$ be a four-dimensional Walker metric \eqref{eqn-2.a} and set the metric components
$g_{11}=a$, $g_{12}=c$ and $g_{22}=b$, where $g_{ij}$ are functions of the Walker coordinates
$(x^1,x^2,x_{1'},x_{2'})$. Then the self-dual Weyl curvature operator takes the form (see~\cite{DR-GR-VL-06})
\begin{equation}\label{W+- matrix form}
    W^+ =
    \left(
    \begin{array}{ccc}
         W_{11}^+ &
         W_{12}^+ &
         W^+_{11} + \frac{\tau}{12}
        \\[0.1in]
        - W_{12}^+&
        \frac{\tau}{6}&
        - W^+_{12}
        \\[0.1in]
        - W^+_{11} - \frac{\tau}{12} &
        - W^+_{12} &
        - W^+_{11} - \frac{\tau}{6}
    \end{array}
    \right),
\end{equation}
where
\begin{equation}\label{W+ 11}
\,
\begin{array}{l}
   W^+_{11} = \frac{1}{12} (
   6 c a_1 b_2 - 6 a_1 b_{1'} - 6 b a_1 c_2  + 12 a_1 c_{2'} - 6 c a_2 b_1 + 6 a_2 b_{2'} 
   \\[0.1in]
   \phantom{W^+_{11} = }
  + 6 b a_2 c_1 +  6 a_{1'} b_1 - 6 a_{2'} b_2 - 12 a_{2'} c_1  + 6 a b_1 c_2 - 6 a b_2 c_1 

   \\[0.1in]
   \phantom{W^+_{11} =  }
  + 12 b_2 c_{1'} - 12 b_{1'} c_2 - a_{11} - 12 c^2 a_{11} - 12 b c a_{12} + 24 c a_{12'} 
   \\[0.1in]
   \phantom{W^+_{11} =  }
   - 3 b^2 a_{22} + 12 b a_{22'} -12 a_{2'2'} - 3 a^2 b_{11} + 12 a b_{11'} - b_{22} 

   \\[0.1in]
   \phantom{W^+_{11} =  }
   - 12 b_{1'1'} + 12 a c c_{11} - 2 c_{12} + 6 a b c_{12} - 24 c c_{11'} - 12 a c_{12'} 
   
    \\[0.1in]
      \phantom{W^+_{11} =  }
   - 12 b c_{21'}  + 24 c_{1'2'}
   ),
\end{array}
\end{equation}
and
\begin{equation}\label{W+ 12}
\begin{array}{l}
   W^+_{12} = \frac{1}{4}  (
   - 2 c a_{11} - b a_{12} + 2 a_{12'} + a b_{12} - 2 b_{21'} + a c_{11} - 2 c c_{12}
   \\[0.1in]
   \phantom{W^+_{12} = \frac{1}{4}  (  }
    - 2 c_{11'} - b c_{22} + 2 c_{22'}
   ).
\end{array}
\end{equation}

Since any anti-self-dual metric is Bach-flat, we proceed as in the proof of Theorem \ref{th:Bach-zero} considering local coordinates $(x^1,x^2)$ on the surface $\Sigma$ such that $T$ is determined  by $T\partial_1=\partial_2$ and $T\partial_2=0$. Since $T$ is parallel, the Christoffel symbols must satisfy \eqref{eq:Christ-T-parallel}, i.e., 
\[
 {}^\affcon\Gamma_{12}^1=0,
	\qquad
	{}^\affcon\Gamma_{12}^2={}^\affcon\Gamma_{11}^1,
	\qquad
	{}^\affcon\Gamma_{22}^1=0,
	\qquad
{}^\affcon\Gamma_{22}^2=0\,.
\]
Next, we analyze the self-dual Weyl curvature operator, which is completely determined by the scalar curvature and its components $W_{11}^+$ and $W_{12}^+$ already described in equations \eqref{W+ 11} and \eqref{W+ 12}. 
The scalar curvature is zero by Theorem \ref{re:Ricci}, and 
$W_{12}^+ = -2\partial_2{}^\affcon\Gamma_{11}^1$, from where it follows that the Ricci tensor $\rho^\affcon$ is symmetric of rank one and recurrent (see Remark \ref{re:recurrente}).
Take local coordinates $(u^1,u^2)$ as in Remark \ref{re:recurrente} so that the only non-zero Christoffel symbol is ${}^u\Gamma_{11}^2$ and 
$T\partial_{u^1}=\partial_{u^2}$, $T\partial_{u^2}=0$.
Finally, we compute the component $W_{11}^+$ given by \eqref{W+ 11} in the coordinates $(u^1,u^2,u_{1^\prime},u_{2^\prime})$ of $T^*\Sigma$, obtaining
\[
\begin{array}{lcl}
W_{11}^+ &=& 
(\partial_2\Phi_{22}+2\partial_{22}{}^u\Gamma_{11}^2)\, u_{2^\prime}
-\frac{1}{2}(\Phi_{22})^2-2\Phi_{22}\partial_2{}^u\Gamma_{11}^2
-\partial_2\Phi_{22}{}^u\Gamma_{11}^2
\\
\noalign{\medskip}
&&
+2\partial_{12}\Phi_{12}-\partial_{22}\Phi_{11}-\partial_{11}\Phi_{22}\,.
\end{array}
\]
Thus $(T^*\Sigma,g_{\affcon,\Phi,T})$ is anti-self-dual if and only if 
\[
\begin{array}{l}
\partial_2\Phi_{22}+2\partial_{22}{}^u\Gamma_{11}^2=0\,,
\\
\noalign{\medskip}
\frac{1}{2}(\Phi_{22})^2+2\Phi_{22}\partial_2{}^u\Gamma_{11}^2
+\partial_2\Phi_{22}{}^u\Gamma_{11}^2
=2\partial_{12}\Phi_{12}-\partial_{22}\Phi_{11}-\partial_{11}\Phi_{22}\,,
\end{array}
\]
from where \eqref{eq:asd-pde} follows. 
\end{proof}

\subsection{Anti-self-dual gradient Ricci solitons}\label{asd-grs}
 Half conformally flat gradient Ricci solitons are locally conformally flat in the Riemannian setting. The same result holds true in neutral signature for non-isotropic ($\|\nabla f\|^2\neq0$) gradient Ricci solitons (see \cite{BVGR, chen-wang}).
 Isotropic half conformally flat gradient Ricci solitons are not necessarily locally conformally flat, and they are realized on Walker manifolds \cite{BVGR}. However, although the self-dual ones were already described in Theorem \ref{th:BVGR1}, no explicit examples of strictly anti-self-dual gradient Ricci solitons were previously reported. 
In order to construct the desired examples we firstly specialize Theorem \ref{th:sd-asd} to get the following anti-self-dual nilpotent Riemannian extensions.

\begin{proposition}\label{prop:4}
Let $(\Sigma,\affcon,T,\Phi)$ be an affine surface equipped with a parallel nilpotent $(1,1)$-tensor field $T$ and a parallel symmetric $(0,2)$-tensor field $\Phi$. If $\rho^\affcon$ is symmetric 
then $(T^*\Sigma,g_{\affcon,\Phi,T})$ is anti-self-dual if and only if $\widehat\omega=0$ and $\widehat{\Phi}=0$, where $\omega$ is the recurrence one-form given by \eqref{eq:recurrence-ricci}.
\end{proposition}

\begin{proof}
If the Ricci tensor $\rho^\affcon$ is symmetric of rank one and $\Phi$ is parallel, then the equations in Theorem \ref{th:sd-asd} reduce to $\widehat{\omega}=0$ and $\widehat{\Phi}=0$, which proves the result. Moreover if $(\Sigma,\affcon)$ is a flat surface, a straightforward calculation shows that anti-self-duality is equivalent to $\widehat{\Phi}=0$, being $\Phi$ a parallel tensor.
\end{proof}

The condition $\widehat{\Phi}=0$ in previous proposition restricts the consideration of Ricci solitons on $T^*\Sigma$ to those originated by an affine gradient Ricci soliton on $\Sigma$ (see Theorem \ref{th:solitons}). An application of Propositions \ref{prop:2} now gives the desired examples of strictly anti-self-dual gradient Ricci solitons. 

\begin{theorem}
\label{prop:11}
Let $(\Sigma,\affcon,T,\Phi)$ be an affine surface equipped with a parallel nilpotent $(1,1)$-tensor field $T$ and a parallel symmetric $(0,2)$-tensor field $\Phi$.
If the Ricci tensor $\rho^\affcon$ is symmetric, then any  $h\in\mathcal{C}^\infty(\Sigma)$ is an affine gradient Ricci soliton with $0=dh(\operatorname{ker}T)$ if and only if the nilpotent Riemannian extension $(T^*\Sigma,g_{\affcon,\widehat{\Phi},T},f=h\circ\pi)$ is an anti-self-dual gradient steady Ricci soliton which is not locally conformally flat.

Moreover, there exist local coordinates $(u^1,u^2)$ on $\Sigma$ so that the only non-zero Christoffel symbol is given by ${}^u\Gamma_{11}^2=\alpha(u^1)+u^2\beta(u^1)$ and the potential function $h(u^1)$ is determined by $h''(u^1)=-2\beta(u^1)$, for any $\alpha,\beta\in\mathcal{C}^\infty(\Sigma)$.
\end{theorem}

\begin{proof}
	First of all, note that $(T^*\Sigma,g_{\affcon,\widehat{\Phi},T},f=h\circ\pi)$ 
	is a gradient Ricci soliton by Proposition \ref{prop:2}. Anti-self-duality follows from Proposition \ref{prop:4}.
Next, since the Ricci tensor is symmetric, Remark \ref{re:recurrente} shows that it is recurrent. Take local coordinates $(u^1,u^2)$ on $\Sigma$ as in Remark \ref{re:recurrente}. 
Then $(\Sigma,\affcon,T)$ admits an affine gradient Ricci soliton with potential function $h$ such that $dh(\operatorname{ker}T)=0$ if and only if $h(u^1,u^2)=S(u^1)$ for some $S\in\mathcal{C}^\infty(\Sigma)$ and 
$(\operatorname{Hes}^\affcon_h+2\rho^\affcon_{sym})(\partial_1,\partial_1)=\partial_{11}h+2\partial_2{}^u\Gamma_{11}^2=0$ becomes
$S''(x^1)=-2\partial_{2}{}^u\Gamma_{11}^2$. Hence the integrability condition $\partial_{22}{}^u\Gamma_{11}^2=0$ is equivalent to ${}^u\Gamma_{11}^2=\alpha(u^1)+u^2\beta(u^1)$ for some $\alpha,\beta\in\mathcal{C}^\infty(\Sigma)$ and $h''(u^1)=-2\beta(u^1)$.
\end{proof}

\section{Conformally Einstein nilpotent Riemannian extensions}\label{se:CE}

A pseudo-Riemannian manifold $(M^n,g)$ is said to be \emph{(locally) conformally Einstein} if every point $p\in M$
has an open neighborhood $\mathcal{U}$ and a positive smooth function $\varphi$ defined on $\mathcal{U}$ such that $(\mathcal{U},\bar{g} =
\varphi^{-2} g)$ is Einstein.
Brinkmann \cite{Brinkmann24} showed that a manifold is conformally Einstein if and only if
the equation
\begin{equation}\label{eq:Conformally Einstein}
(n-2)\operatorname{Hes}_\varphi +\varphi\,\rho- \frac{1}{n}\{(n-2)\Delta\varphi+\varphi\,\tau\}g=0
\end{equation}
has a positive solution. Besides its apparent simplicity, the integration of the conformally Einstein equation is surprisingly difficult (see \cite{Kuhnel-Rademacher} and references therein for more information).
It was shown in \cite{GN, Kozameh-Newman-Tod85} that any four-dimensional conformally Einstein manifold satisfies  
\begin{equation}\label{eq:18-b}
(i) \quad \mathfrak{C}+W(\cdot,\cdot,\cdot,\nabla \sigma)=0,\qquad \qquad
(ii) \quad\bach=0.
\end{equation}
where the conformal metric is given by $\bar{g}=e^{2\sigma}g$. 

Conditions (i)-(ii) above are also sufficient to be conformally Einstein if $(M,g)$ is \emph{weakly-generic} (i.e., the Weyl tensor viewed as a map $TM\rightarrow\bigotimes^3 TM$ is injective). 
Since nilpotent Riemannian extensions are not weakly-generic (see the expression of $W^-$ in the proof of Theorem \ref{th:sd-asd}), we will analyze the conformally Einstein equation \eqref{eq:Conformally Einstein}, seeking for solutions on  nilpotent Riemannian extensions $(T^*\Sigma,g_{\affcon,\Phi,T})$.

\begin{theorem}\label{th:conformally-Einstein}
Let $(\Sigma,\affcon,T)$ be a torsion free affine surface equipped with a parallel nilpotent $(1,1)$-tensor field $T$. Then  any solution of \eqref{eq:Conformally Einstein} is of the form
$\varphi=\iota X+\phi\circ\pi$ for some vector field $X$ on $\Sigma$ such that $X\in\operatorname{ker}T$ and $\operatorname{tr}(\affcon X)=0$.

Moreover
$(T^*\Sigma,g_{\affcon,\Phi,T})$ is conformally Einstein if and only if one of the following holds:
\begin{itemize}
\item[(i)]
The conformally Einstein equation \eqref{eq:Conformally Einstein} admits a solution $\varphi=\phi\circ\pi$ for some $\phi\in\mathcal{C}^\infty(\Sigma)$ with $d\phi(\operatorname{ker}T)=0$, and the deformation tensor $\Phi$ is determined by
$\phi\,\widehat{\Phi}+2(\operatorname{Hes}^\affcon_\phi+\phi\,\rho^\affcon_{sym})=0$.
\item[(ii)]
The conformally Einstein equation \eqref{eq:Conformally Einstein} admits a solution $\varphi=\iota X+\phi\circ\pi$ for some $\phi\in\mathcal{C}^\infty(\Sigma)$ and some non-zero vector field $X$ on $\Sigma$ such that $X\in\operatorname{ker}T$ and $\operatorname{tr}(\affcon X)=0$. 

In this case, the Ricci tensor $\rho^\affcon$ is symmetric of rank one and recurrent. 
Moreover there are local coordinates $(u^1,u^2)$ on $\Sigma$ so that   $\varphi(u^1,u^2$, $u_{1^\prime},u_{2^\prime})=\kappa u_{2^\prime}+\phi(u^1,u^2)$ is a solution of \eqref{eq:Conformally Einstein} if and only~if
$$
\begin{array}{l}
d\phi(T\partial_1)=\tfrac{\kappa}{2} \Phi(T\partial_1,T\partial_1)\,,\\
\noalign{\medskip}
\operatorname{Hes}^D_\phi(\partial_1,\partial_1)+\phi\,\rho^\affcon(\partial_1,\partial_1)
=-\frac{1}{2}(\phi+2\kappa\,{}^u\Gamma_{11}^2)\Phi(T\partial_1,T\partial_1)
\\
\noalign{\medskip}
\phantom{2\operatorname{Hes}^D_\phi(\partial_1,\partial_1)}
+\tfrac{\kappa}{2}\left\{
2(\affcon_{\partial_1}\Phi)(T\partial_1,\partial_1)-(\affcon_{T\partial_1}\Phi)(\partial_1,\partial_1)
\right\}\,.
\end{array}
$$
\end{itemize}
\end{theorem}

\begin{proof}
Let $(x^1,x^2)$ be local coordinates on $\Sigma$ so that $T\partial_1=\partial_2$, $T\partial_2=0$, and consider the induced coordinates $(x^1,x^2,x_{1'},x_{2'})$ on $T^*\Sigma$. Since $T$ is parallel, and we obtain directly from Equation \eqref{eq:nablaT} that 
\[
	{}^\affcon\Gamma_{12}^1=0,\quad
	{}^\affcon\Gamma_{12}^2={}^\affcon\Gamma_{11}^1,\quad
	{}^\affcon\Gamma_{22}^1=0,\quad
	{}^\affcon\Gamma_{22}^2=0\,.
\]
In order to analyze the conformally Einstein equation \eqref{eq:Conformally Einstein} consider the symmetric $(0,2)$-tensor field $\ce=2\operatorname{Hes}_\varphi +\varphi\,\rho- \frac{1}{4}\{2\Delta\varphi+\varphi\,\tau\}g$ and set $\ce=0$.
Let $\ce_{ij}=\ce(\partial_i,\partial_j)$ and let $\varphi\in\mathcal{C}^\infty(T^*\Sigma)$ be a solution of \eqref{eq:Conformally Einstein}. Then one computes
\[
\ce_{33}=2 \partial_{1'1'}\varphi,\quad
\ce_{34}=2 \partial_{1'2'}\varphi,\quad
\ce_{44}=2 \partial_{2'2'}\varphi\,,
\]
to show that any solution of \eqref{eq:Conformally Einstein} must be of the form \begin{equation}\label{eq:conformal function Walker}
\varphi(x^1,x^2,x_{1'},x_{2'})=
A(x^1,x^2) x_{1'}
+B(x^1,x^2)x_{2'}
+\psi(x^1,x^2)\,,
\end{equation}
for some smooth functions $A$, $B$ and $\psi$ depending only on the coordinates $(x^1,x^2)$.
This shows that any solution of the conformally Einstein equation on $(T^*\Sigma,g_{\affcon,\Phi,T})$ is of the form
$\varphi=\iota X+\psi\circ\pi$, where $\iota X$ is the evaluation of a vector field $X=A\partial_1+B\partial_2$ on $\Sigma$, $\psi\in\mathcal{C}^\infty(\Sigma)$ and $\pi:T^*\Sigma\rightarrow\Sigma$ is the projection. 

Now, the conformally Einstein condition given in Equation \eqref{eq:Conformally Einstein} can be expressed in matrix form as follows:
\begin{equation}\label{eq:ce}
(\ce_{ij})\!=\!\left(\!
\begin{array}{cccc}
		\ce_{11} & \ce_{12} 
		& 
		\partial_1A-\partial_2B 
		&
		2(
		{}^\affcon\Gamma_{11}^2 A
		+{}^\affcon\Gamma_{11}^1 B
		+\partial_1 B
		-A x_{2'}
		)
		\\
		\ast & \ce_{22} 
		& 
		2\partial_2A 
		&
		-\partial_1A+\partial_2B
	    \\
	    \ast & \ast & 0 & 0
		\\
		\ast & \ast & \ast & 0
\end{array}
\!\right) 
\end{equation}
where positions with $\ast$ are not written since the matrix is symmetric, and where
\medskip

\noindent$
\begin{array}{rcl}
\ce_{11}
&=& 
-(
   \partial_1A
   -\partial_2B
   -4{}^\affcon\Gamma_{11}^1 A
   )x_{2'}^2

\\[0.1in]

&+&
\{
   A \Phi_{22}
   +2(
      \partial_{11}A
      -{}^\affcon\Gamma_{11}^2\partial_2A
      +{}^\affcon\Gamma_{11}^1\partial_2B
      +A \partial_2{}^\affcon\Gamma_{11}^2 
      -B \partial_2{}^\affcon\Gamma_{11}^1 
   )
   \}x_{1'}

\\[0.1in]

&-&
\{
   B \Phi_{22}
   +2 A \Phi_{12}
   -2(
        \partial_{11}B
        +{}^\affcon\Gamma_{11}^2 \partial_1A
        -{}^\affcon\Gamma_{11}^1 \partial_1B
        
        \\[0.1in]
        &&
        
        +(\partial_1{}^\affcon\Gamma_{11}^2-2{}^\affcon\Gamma_{11}^1{}^\affcon\Gamma_{11}^2) A
        +(\partial_1{}^\affcon\Gamma_{11}^1-2({}^\affcon\Gamma_{11}^1)^2) B
        +\partial_2\psi
     )
   \}x_{2'}

\\[0.1in]

&+&
2 \partial_2A x_{1'}x_{2'}

\\[0.1in]

&-&
(
   \partial_1A+\partial_2B 
   )\Phi_{11}
+2(
   {}^\affcon\Gamma_{11}^2A+{}^\affcon\Gamma_{11}^1B
   )\Phi_{12}
+(
   2{}^\affcon\Gamma_{11}^2 B+\psi
   )\Phi_{22}

\\[0.1in]

&-&
A \partial_1\Phi_{11}
+B \partial_2\Phi_{11}
-2B \partial_1\Phi_{12}

\\[0.1in]

&+&
2\partial_{11}\psi
-2{}^\affcon\Gamma_{11}^1\partial_1\psi
-2{}^\affcon\Gamma_{11}^2\partial_2\psi
-2(\partial_1{}^\affcon\Gamma_{11}^1-\partial_2{}^\affcon\Gamma_{11}^2)\psi \,,

\end{array}
$

\bigskip

\noindent$
\begin{array}{rcl}
\ce_{12}
&=& 
2(
   \partial_{12}A
   -{}^\affcon\Gamma_{11}^1\partial_2A
   +A \partial_2{}^\affcon\Gamma_{11}^1 
   )x_{1'}
   
\\[0.1in]

&+&
2(
   \partial_{12}B
   +{}^\affcon\Gamma_{11}^1\partial_1A
   + A \partial_2{}^\affcon\Gamma_{11}^2
   )x_{2'}
   
\\[0.1in]

&-&
(\partial_1A+\partial_2B)\Phi_{12}
+2{}^\affcon\Gamma_{11}^1 B \Phi_{22}
-A \partial_2\Phi_{11}
-B \partial_1\Phi_{22}

\\[0.1in]

&+&
2\partial_{12}\psi
-2{}^\affcon\Gamma_{11}^1 \partial_2\psi \,,
\end{array}
$

\bigskip

\noindent$
\begin{array}{rcl}
\ce_{22}
&=& 
2\partial_{22}A x_{1'}
+2(
    \partial_{22}B+2 A \partial_{2}{}^\affcon\Gamma_{11}^1
    )x_{2'}

\\[0.1in]

&-&
(\partial_1A
   +\partial_2B
   +2{}^\affcon\Gamma_{11}^1 A
   )\Phi_{22}
-2A\partial_2\Phi_{12}
+A \partial_1\Phi_{22}
-B \partial_2\Phi_{22}
\\[0.1in]

&+&
2\partial_{22}\psi \,.
\end{array}
$

\bigskip
First, we use component 
$
\ce_{14}=2(
{}^\affcon\Gamma_{11}^2 A
+{}^\affcon\Gamma_{11}^1 B
+\partial_1 B
-A x_{2'}
)
$ in Equation~\eqref{eq:ce}; note that   $\partial_{2'}\ce_{14}=-2A$, and therefore 
$A(x^1,x^2)=0$,
which shows that $X\in\operatorname{ker}T$.
Now component $\ce_{13}$ in Equation~\eqref{eq:ce} gives $\partial_{2}B=0$, which implies  
$B(x^1,x^2)=P(x^1)$
for some smooth function $P$ depending only on the coordinate $x^1$, i.e., the vector field $X=B\partial_{2}$ satisfies $\operatorname{tr}(\affcon X)=0$.

At this point,  the conformal function $\varphi$ has  the coordinate expression 
\[
\varphi(x^1,x^2,x_{1'},x_{2'})=P(x^1) x_{2'}+\psi(x^1,x^2)
\]
and the possible non-zero components in  Equation~\eqref{eq:ce} are $\ce_{11}$, $\ce_{12}$, $\ce_{22}$ and $\ce_{14}$.
Considering the component $\ce_{14}=2(
P'(x^1)+{}^\affcon\Gamma_{11}^1(x^1,x^2) P(x^1))$, we distinguish two cases depending on whether the function $P$ vanishes identically or not.
Indeed, if $P(x^1)$ is a solution of the equation $\ce_{14}=0$, then $\partial_1\left(P(x^1)e^{\int{}^\affcon\Gamma_{11}^1(x^1,x^2) dx^1}\right)$ $=$ $e^{\int{}^\affcon\Gamma_{11}^1(x^1,x^2) dx^1}\left\{ P'(x^1)+P(x^1){}^\affcon\Gamma_{11}^1(x^1,x^2) \right\}$ $=$ $0$, which shows that 
$P(x^1)e^{\int{}^\affcon\Gamma_{11}^1(x^1,x^2) dx^1}=\mathcal{Q}(x^2)$ for some smooth function $\mathcal{Q}(x^2)$. Now, if the function $\mathcal{Q}(x^2)$ vanishes at some point, then $P(x^1)=0$ at each point. Otherwise, if $\mathcal{Q}(x^2)\neq 0$ at each point, so is $P(x^1)$.

 First, suppose that $P(x^1)\equiv 0$, and hence $\varphi=\psi\circ\pi$. In this case, component $\ce_{22}$ in  Equation~\eqref{eq:ce} yields $\partial_{22}\psi=0$, which implies 
$\psi(x^1,x^2)=Q(x^1) x^2+\phi(x^1)$
for some smooth functions $Q$ and $\phi$ depending only on the coordinate $x^1$. Now, the only components in   Equation~\eqref{eq:ce} which could be non-null are 
\[
\begin{array}{rcl}
\ce_{11}
&=& 
2Q x_{2'}
+(Q\Phi_{22}
  +2 Q''
  -2{}^\affcon\Gamma_{11}^1 Q'
  -2(\partial_1{}^\affcon\Gamma_{11}^1-\partial_2{}^\affcon\Gamma_{11}^2) Q
  )x_2
  
\\[0.1in]
&+&

\phi\Phi_{22}
+2\phi''
-2{}^\affcon\Gamma_{11}^1\phi'
-2(\partial_1{}^\affcon\Gamma_{11}^1-\partial_2{}^\affcon\Gamma_{11}^2)\phi
-2{}^\affcon\Gamma_{11}^2 Q\,,

\\[0.15in]

\ce_{12} &=& 2(Q'-{}^\affcon\Gamma_{11}^1 Q)\,.

\end{array}
\]
Now, $\partial_{2'}\ce_{11}=2 Q$, implies   $Q=0$, thus showing that $d\varphi(\operatorname{ker}T)=0$. Then $\ce_{12}=0$ and the component $\ce_{11}$ reduces to
\[
\ce_{11} = 
\phi\Phi_{22}
+2\phi''
-2{}^\affcon\Gamma_{11}^1\phi'
-2(\partial_1{}^\affcon\Gamma_{11}^1-\partial_2{}^\affcon\Gamma_{11}^2)\phi\,.
\]
Since $\varphi(x^1,x^2,x_{1'},x_{2'})= \phi(x^1)$, $\phi$ must be non-null and we obtain that $\ce_{11}=0$ is equivalent to
\[
\begin{array}{rcl}
\Phi_{22} &=&
-\tfrac{2}{\phi} \left\{
\phi''-{}^\affcon\Gamma_{11}^1\phi'
-(\partial_1{}^\affcon\Gamma_{11}^1-\partial_2{}^\affcon\Gamma_{11}^2)\, \phi
\right\}
\,,
\\
\noalign{\medskip}
&=&-\tfrac{2}{\phi}\left\{\operatorname{Hes}^\affcon_\phi(\partial_1,\partial_1)+\phi\,\rho^D_{sym}(\partial_1,\partial_1)\right\}\,,
\end{array}
\]
from where (i) is obtained.

\medskip

Finally, we analyze the case in which the function $P(x^1)$ does not vanish identically. 
Since  $\ce_{14}=2(P'(x^1)+{}^\affcon\Gamma_{11}^1(x^1,x^2) P(x^1))$, we have 
$\partial_2{}^\affcon\Gamma_{11}^1=0$.
Now it follows from Remark \ref{re:recurrente} that the Ricci tensor $\rho^\affcon$ is symmetric of rank one and recurrent. Specialize the local coordinates $(u^1,u^2)$ on $\Sigma$ so that the only non-zero Christoffel symbol of $\affcon$ is ${}^u\Gamma_{11}^2(u^1,u^2)$ and $T\partial_{u^1}=\partial_{u^2}$, $T\partial_{u^2}=0$.
Then any solution of the conformally Einstein equation takes the form
\[
\varphi(u^1,u^2,u_{1'},u_{2'})=\mathcal{A}(u^1) u_{2'}+\phi(u^1,u^2).
\]
Now, considering the component $\ce_{41}$ of the conformally Einstein equation in the new coordinates $(u^1,u^2)$, one has $\ce_{41}=2\mathcal{A}'(u^1)$, which shows that  
$\varphi(u^1,u^2,u_{1'},u_{2'})=\kappa u_{2'}+\phi(u^1,u^2)$ for some $\kappa\neq 0$.
Considering now the component
\[
\begin{array}{l}
\ce_{11}=
(2\partial_2\phi-\kappa\Phi_{22})u_{2^\prime} +2\partial_{11}\phi-2\partial_2\phi {}^u\Gamma_{11}^2
\\
\noalign{\medskip}
\phantom{\ce_{11}=}
+2\phi\partial_2 {}^u\Gamma_{11}^2+\phi\Phi_{22}+2\kappa\Phi_{22}{}^u\Gamma_{11}^2
+\kappa\partial_2\Phi_{11}-2\kappa\partial_1\Phi_{12}
\end{array}
\]
it follows that the conformally Einstein equation reduces to  
\[
\begin{array}{rcl}
\kappa\Phi_{22}&=&2\partial_2\phi, \\
(\phi+2\kappa {}^u\Gamma_{11}^2)\Phi_{22}&=&-2(\operatorname{Hes}_\phi^\affcon(\partial_{u^1},\partial_{u^1})+ \phi\rho^\affcon(\partial_{u^1},\partial_{u^1}))\\
\noalign{\medskip}
&&
+\kappa(2\partial_1\Phi_{12}-\partial_2\Phi_{11})\,,
\end{array}
\]
from where (ii) is obtained.
\end{proof}

\section{Examples}

\subsection{Nilpotent Riemannian extensions with flat base}
\label{re:flat-11}
Let $(\Sigma,\affcon)$ be a flat torsion free affine surface. Take local coordinates on $\Sigma$ so that all Christoffel symbols vanish. Let $T$ be a parallel nilpotent $(1,1)$-tensor field. Since $T$ is parallel, its components $T_i^j$ are necessarily constant on the given coordinates. Hence one may further specialize the local coordinates $(x^1,x^2)$, by using a linear transformation, so that  $T\partial_1=\partial_2$, $T\partial_2=0$ and all the Christoffel symbols ${}^\affcon\Gamma_{ij}^k$ remain identically zero.
Now Theorem \ref{th:Bach-zero} shows that 
\emph{$(T^\ast\Sigma, g_{\affcon,\Phi,T})$ is Bach-flat for any symmetric $(0,2)$-tensor field $\Phi$ on $\Sigma$}. Moreover it follows from Theorem \ref{th:solitons} that 
\emph{$(T^\ast\Sigma, g_{\affcon,\Phi,T},f=h\circ\pi)$ is a steady gradient Ricci soliton for any $h\in\mathcal{C}^\infty(\Sigma)$ with $dh\circ T=0$ and any symmetric $(0,2)$-tensor field $\Phi$ such that $\Phi_{22}(x^1,x^2)=-h''(x^1)$}.

Further note from Remark \ref{re:4-3} that the steady gradient Ricci soliton $(T^\ast\Sigma, g_{\affcon,\Phi,T}$, $f=h\circ\pi)$ satisfies $\mathbb{D}=0$. Moreover, since $\Phi_{22}=-h''(x^1)$, one has that $(T^\ast\Sigma, g_{\affcon,\Phi,T})$ is in the conformal class of an Einstein metric (just considering the conformal metric $\bar g=\phi^{-2}g_{\affcon,\Phi,T}$ determined by the equation
$\phi''(x^1)-\frac{1}{2}\phi(x^1)h''(x^1)=0$).

\begin{remark}\rm
Set $\Sigma=\mathbb{R}^2$ with usual coordinates $(x^1,x^2)$ and put $T\partial_1=\partial_2$, $T\partial_2=0$. For any  any smooth function $h(x^1)$ consider the deformation tensor $\Phi$ given by $\Phi_{22}(x^1,x^2)=-h''(x^1)$ (the other components being zero). 
Then, the non zero Christoffel symbols of $g_{\affcon,\Phi,T}$ are given by
$$
\Gamma_{11}^2=-x_{2'}=-\Gamma_{12'}^{1'}\,,\quad
\Gamma_{11}^{2'}=-h''(x^1)x_{2'}\,,\quad
\Gamma_{12}^{2'}=-\frac{1}{2}h^{(3)}(x^1)=-\Gamma_{22}^{1'}\,.
$$
Hence a curve $\gamma(t)=(x^1(t),x^2(t),x_{1'}(t),x_{2'}(t))$ is a geodesic if and only if
$$
\begin{array}{l}

\ddot{x}^1(t)=0\,,\qquad\qquad 
\ddot{{x}}^2(t)-\, x_{2'}(t)\,\dot{x}^1(t)^2=0\,,
\\
\noalign{\medskip}
\ddot{x}_{1'}(t)+2\, x_{2'}(t)\,\dot{x}^1(t)\dot{x}_{2'}(t)+\frac{1}{2}\,h^{(3)}(x^1(t))\,\dot{x}^2(t)^2=0\,,
\\
\noalign{\medskip}
\ddot{x}_{2'}(t)-h''(x^1(t))\, x_{2'}(t)\,\dot{x}^1(t)^2-\,h^{(3)}(x^1(t))\,\dot{x}^1(t)\,\dot{x}^2(t)=0\,.
\end{array}
$$
Thus ${x}^1(t)=at+b$ for some $a,b\in\mathbb{R}$ and
$$
\begin{array}{l}
\ddot{{x}}^2(t)-\,a^2\, x_{2'}(t)=0\,,
\\
\noalign{\medskip}
\ddot{x}_{2'}(t)-h''(at+b)\,a^2\, x_{2'}(t)-\,h^{(3)}(at+b)\,a\,\dot{x}^2(t)=0\,.
\\
\noalign{\medskip}
\ddot{x}_{1'}(t)+2a\, x_{2'}(t)\,\dot{x}_{2'}(t)+\frac{1}{2}\,h^{(3)}(at+b)\,\dot{x}^2(t)^2=0\,,

\end{array}
$$
Now the first two equations above are linear and thus $x^2(t)$ and $x_{2'}(t)$ are globally defined. 
Finally, since 
$\ddot{x}_{1'}(t)+2a\, x_{2'}(t)\,\dot{x}_{2'}(t)+\frac{1}{2}\,h^{(3)}(at+b)\,\dot{x}^2(t)^2=0$ is also linear on ${x}_{1'}(t)$, one has that geodesics are globally defined. 

Then it follows from Theorem \ref{th:solitons}  that 
$(T^\ast\mathbb{R}^2, g_{\affcon,\Phi,T},f=h\circ\pi)$ is a geodesically complete steady gradient Ricci soliton, which is conformally Einstein by Theorem \ref{th:conformally-Einstein}.

\end{remark}

%

\subsection{Nilpotent Riemannian extensions with non recurrent base}

\label{re:solitones-no-CE} 
Let $(T^*\Sigma,g_{\affcon,\Phi,T},f=h\circ\pi)$ be a non-trivial Bach-flat steady gradient Ricci soliton as in Theorem \ref{th:solitons}. Further assume that the Ricci tensor $\rho^\affcon$ is non symmetric, i.e.,  $\rho^\affcon_{sk}\neq 0$ (equivalently $\partial_2{}^\affcon\Gamma_{11}^1\neq 0$ as shown in the proof of Proposition \ref{prop:1}). Then it follows from Theorem \ref{th:sd-asd} that $(T^*\Sigma,g_{\affcon,\Phi,T})$ is not half conformally flat. 

Theorem \ref{th:conformally-Einstein} shows that $(T^*\Sigma,g_{\affcon,\Phi,T})$ is conformally Einstein if and only if there exists a positive $\phi\in\mathcal{C}^\infty(\Sigma)$ with $d\phi\circ T =0$ such that
$\phi\,\widehat{\Phi}+2(\operatorname{Hes}^\affcon_\phi+\phi\,\rho^\affcon_{sym})=0$. 
Hence it follows from Theorem \ref{th:solitons} that $\operatorname{Hes}^\affcon_h=\frac{2}{\phi}\operatorname{Hes}^\affcon_\phi$, which means
$(2\frac{\phi'}{\phi}-h'){}^\affcon\Gamma_{11}^1=2\frac{\phi''}{\phi}-h''$.
Taking derivatives with respect to $x^2$ and, since $\partial_2{}^\affcon\Gamma_{11}^1\neq 0$, the equation above splits into
$$
\frac{2\phi'}{\phi}-h'=0\,,\qquad  \mbox{and}\qquad 2\frac{\phi''}{\phi}-h''=0\,,
$$ 
which only admits constant solutions. Summarizing the above one has the following:
\emph{Let $(\Sigma,\affcon,T)$ be an affine surface with non-symmetric Ricci tensor (i.e., $\rho^\affcon_{sk}\neq 0$). Then any Bach-flat gradient Ricci soliton $(T^*\Sigma,g_{\affcon,\Phi,T},f=h\circ\pi)$ is neither half conformally flat nor conformally Einstein.}

\subsection*{Acknoledgments}
It is a pleasure to acknowledge useful conversations on this subject with Professor P. Gilkey. 

Research partially supported by projects GRC2013-045, MTM2013-41335-P, MTM2016-75897-P  and EM2014/009 with FEDER funds (Spain).

\end{document}